\documentclass[10pt, reqno, a4paper]{amsart}
\usepackage{amssymb,latexsym,amsmath,amsfonts, amsthm}
\usepackage{latexsym}
\usepackage[numbers,sort&compress]{natbib}
\usepackage{color}
\usepackage[mathscr]{eucal}
\usepackage{comment}
\usepackage{dsfont}
\newcommand{\mc}[1]{\mathcal{#1}}
\addtocontents{toc}{\protect\setcounter{tocdepth}{1}}

\usepackage[linkcolor=red, citecolor=magenta]{hyperref}
\hypersetup{colorlinks=true, urlcolor=blue}
\usepackage[nameinlink]{cleveref}

\usepackage[top=2.7cm,bottom=2.7cm,left=2.7cm,right=2.7cm]{geometry}

\numberwithin{equation}{section}

\theoremstyle{definition}
\newtheorem{definition}{Definition}[section]

\theoremstyle{remark}
\newtheorem{remark}[definition]{Remark}
\theoremstyle{plain}
\newtheorem{theorem}[definition]{Theorem}
\newtheorem{lemma}[definition]{Lemma}
\newtheorem{proposition}[definition]{Proposition}

\newcommand{\norm}[1]{\left\|#1\right\|}
\newcommand{\abs}[1]{\left|#1\right|}

\newcommand{\ip}[2]{\left<{#1},{#2}\right>}

\newcommand{\zahl}{\mathbb{Z}}

\newcommand{\wphi}{\widehat{\phi}}

\newcommand\Iint{\int\limits_{0}^{T}\int\limits_{0}^{1}}

\newcommand{\vphi}{\varphi}

\newcommand \F{\mathcal F}

\newcommand\Y{\mathcal{Y}}
\newcommand\V{\mathcal{V}}

\newcommand\C{\mathcal{C}}

\newcommand{\iintQ}{\iint_{Q_T}}

\date{\today}

\begin{document}
	
	\title[Controllability of the KS-KdV-elliptic system]{Local null-controllability of a system   coupling Kuramoto-Sivashinsky-KdV and elliptic equations}
	\author[K. Bhandari, and S. Majumdar]{Kuntal Bhandari$^*$, and Subrata Majumdar$^{\dagger}$}
	\thanks{$^*$Institute of Mathematics, Czech Academy of Sciences, \v{Z}itn\'a 25, Prague 1, Czech Republic, e-mail: bhandari@math.cas.cz (corresponding author)}
	\thanks{$^\dagger$Indian Institute of Science Education and Research Kolkata, Campus road, Mohanpur, West Bengal 741246, India; e-mail: sm18rs016@iiserkol.ac.in}
	
	\begin{abstract}
	This paper deals with the null-controllability of a system of {\em mixed parabolic-elliptic pdes} at any given time $T>0$. More precisely, we consider the \textit{Kuramoto-Sivashinsky--Korteweg-de Vries equation} coupled with a second order elliptic equation posed in the interval $(0,1)$. We first show that the linearized system is globally null-controllable by means of a  localized interior control acting on either the KS-KdV or the elliptic equation. Using the \textit{Carleman approach}, we provide the existence of a control with the explicit cost $Ce^{C/T}$ with some constant $C>0$ independent in $T$.  Then, applying the source term method developed in \cite{Tucsnak-nonlinear}, followed by the \textit{Banach fixed point theorem}, we conclude the small-time local null-controllability result of the nonlinear systems. 
	\end{abstract}

	\maketitle
	
	
	\section{Introduction and main results}

	\subsection{Bibliographic comments and motivations}
	The controllability of \textit{coupled partial differential equations} has gained immense interest to the control communities since the past few decades. Due to the significant importance of the parabolic models in physical, chemical and biological sciences, several methods have been developed to study their controllability properties.   We first address the pioneer work by  Fattorini and Russell \cite{FR71, FR75} where the authors utilized the spectral strategy, i.e., the so-called \textit{moments method} to study the controllability (distributed or boundary) of 1D linear parabolic equations. Regarding the distributed controllability of parabolic pdes in dimension higher than $1$, we refer the well-known {\em Lebeau-Robbiano spectral method} \cite{LR95} (see also \cite{JR12}) which is based on a global elliptic Carleman estimate. 
	Next, we must mention that the {\em global Carleman estimate}, established by Fursikov and Imanuvilov in \cite{Fur-Ima} has been intensively used  to study the controllability of parabolic equations and systems. In this context, we refer the work \cite{Cara-Guerrero} by E. Fern\'andez-Cara and S. Guerrero where a global Carleman estimate has been established to study the controllability properties of some linear and certain semilinear parabolic equations. Later on, the controllability of a system of two parabolic equations with only one control force  has been proved by S. Guerrero \cite{G07}. 
	
	 However, in  most cases the Carleman approach is  inefficient to handle the boundary controllability of the coupled parabolic systems. In fact, the boundary controllability for such systems is no longer equivalent with distributed controllability as it has been observed for instance in \cite{Cara-Burgos-Teresa}. Due to this reason, most available results are in 1D and mainly based on the moments method. Among the fewest,  we refer for instance, \cite{KBBT11b}, \cite{ABBTc}, \cite{ABBT16}, \cite{Bhandari21a}.  For the multi-D situation, we address the work \cite{BBB14} where the boundary controllability results are obtained in the cylindrical geometries.   
	
In the direction of controllability of nonlinear parabolic systems, we mention the work by E. Fern\'andez-Cara and E. Zuazua \cite{FC05} where the authors proved small-time global null-controllability of semilinear heat equations with the growth of nonlinear function $|f(r)| \leq |r| \ln^{3/2}(1+|r|)$, see also \cite{Barbu} by V. Barbu.  Very recently,  the large-time global null-controllability of semilinear heat equations has been obtained in \cite{Kevin} for nonlinearities $f$ growing slower than $|r|\ln^2(1+|r|)$ under the sign condition:  $f(r)>0$ for $r>0$ and $f(r)<0$ for $r<0$. 
	   
	   \smallskip

	Another fascinating area in control theory is the study of   controllability properties of the fourth order parabolic equation:  $u_t+\gamma_1 u_{xxxx}+ \gamma_2 u_{xx}+ uu_x=0$ ($\gamma_1,\gamma_2>0$), the so-called   \textit{Kuramoto-Sivashinsky} (KS) equation. This  has been  popularly studied in the past decades,  such as \cite{CE11}, \cite{Zhou}, \cite{Cer17}, \cite{Taka17}.  In particular, the authors in \cite{Cer17}  proved that the linear KS equation with Neumann boundary conditions is null-controllable by a  control acting in some open subset of the domain. The author in \cite{Taka17} proved the boundary local null-controllability of the KS equation by utilizing the source term method (see \cite{Tucsnak-nonlinear})  followed by  the Banach fixed point argument where a suitable control cost $C e^{C/T}$ of the linearized model plays  the crucial role.   Similar strategy has  been applied in \cite{hernandezsantamaria:hal-03090716} to study the boundary local null-controllability of a simplified {\em stabilized KS} system (see below about this system). 
	In this regard, we must mention that the nonlinearity $uu_x$ in the KS system also appears  in the Burgers equation: $u_t - u_{xx} + uu_x=0$, and  concerning the controllability results of  this equation we address the following works: \cite{CG07}, \cite{SGI07}, \cite{OG07}.  
	
	\smallskip 
	
To include dispersive and extra dissipative effects in the KS equation, a coupled system containing fourth (Kuramoto-Sivashinsky-Korteweg-de Vries, in short KS-KdV equation) and second order parabolic (heat) equations, under the name of \textit{stabilized Kuramoto-Sivashinsky} system 
	\begin{align}\label{eq:KS_kDV}
		\begin{cases}
			u_t + \gamma_1 u_{xxxx}  + u_{xxx} + \gamma_2 u_{xx} + u u_x =  v_x, \\
			v_t -\Gamma v_{xx} +  c v_x = b u_x,  
		\end{cases}
	\end{align}
	has appeared for instance in \cite{PhysRevE.64.046304}. The controllability of such system has been considered   in several works, namely  \cite{CE10}, \cite{CE12}, \cite{Cerpa-Mercado-Pazoto}, \cite{CE16}, \cite{kumar2022null}.   E. Cerpa, A. Mercado and A. F. Pazoto in \cite{Cerpa-Mercado-Pazoto} demonstrated the  local null-controllability of \eqref{eq:KS_kDV} by means of localized interior control acting in the KS-KdV equation. Then,  E. Cerpa along with N. Carre\~{n}o \cite{CE16} proved the local null-controllability to the trajectories with a localized interior control exerted in the heat equation. Two different types of Carleman estimates for the linearized model followed by the \textit{inverse mapping theorem} have been implemented to  conclude their controllability results. We also point out that recently, the controllability of KS-KdV-transport model has been discussed in \cite{majumdar:hal-03695906}. Finally, it is worth mentioning that an insensitizing control problem for the stabilized KS system has been explored in \cite{KV}. 
	  
	\smallskip
	
Study of the parabolic-elliptic model, namely
	\begin{equation}
		\begin{cases}
			u_t-u_{xx}=au+bv,\\
			-v_{xx}=cu+dv,
		\end{cases}
	\end{equation}
 enlarges the literature of the control of pdes and such type of models  significantly appear in several  physical, chemical and biological  situations: as for example,   the prey-predator interaction, tumor growth therapy etc.  One may refer to the  works \cite{FC13}, \cite{GB14}, \cite{hernandez2021boundary}, \cite{Lim19} concerning the controllability issues  of such systems. 
	
	In the work \cite{FC13}, the authors proved the local null-controllability of a  semilinear parabolic-elliptic equation  by means of a localized interior control acting in  one of  the two equations. Similar question has been addressed in \cite{GB14} for the chemotaxis system of parabolic-elliptic type which is specially known as \textit{Keller-Segel} model. Both papers utilized the Kakutani's fixed-point theorem to handle the  nonlinearity.  In \cite{Lim19}, a local interior null-controllability of such system with local and nonlocal nonliearity has been established using appropriate Carleman estimates and fixed point argument. The paper \cite{FC16} dealt with the same issue for a system coupling two parabolic models and an elliptic model with one scalar control force, locally supported in space. 
	
	These works principally motivate us to study the null-controllability of a parabolic-elliptic system of coupled Kuramoto-Sivashinsky-Korteweg-de Vries equation and a second order elliptic equation. In fact, the coupled system of KS--elliptic equations can be seen as nonlocal KS equation.  In this regard, we mention the work \cite{Ioakim} where the author studied  the analyticity properties of solutions of the nonlocal KS equations arising	
		in the problem of a perfectly conducting thin-film flow down an inclined plane in the presence of an electric field.
	We also quote the work  \cite{Duan},  where the dynamical behavior of some certain non-local KS system has been studied.  However,  we are mostly interested in the mathematical analysis    from the view point of null-controllability of the mentioned systems by means of a single localized control.  
	
	\subsection{The systems under study}
	Let $T>0$ be any finite time and define $Q_T: = (0,T)\times (0,1)$. We also take any non-empty open set $\omega\subset (0,1)$.
Then, we  consider the following  control problems 
	\begin{equation}
		\begin{cases}
			\label{nonlinear-KS}
			u_t+\gamma_1 u_{xxxx}+u_{xxx}+\gamma_2 u_{xx} + uu_x = a v+ \chi_{\omega} h, &\text{in } Q_T,\\
			-v_{xx}+ c v= b u, & \text{in } Q_T, 
		\end{cases} 
	\end{equation}
with the constants $(a, b)\in \mathbb R^2$ such that $b\neq 0$; or 
\begin{equation}
	\begin{cases}
		\label{nonlinear-elliptic}
		u_t+\gamma_1 u_{xxxx}+u_{xxx}+\gamma_2 u_{xx} + uu_x = a v, &\text{in } Q_T,\\
		-v_{xx}+ c v= b u + \chi_{\omega} h,  & \text{in } Q_T, 
	\end{cases} 
\end{equation}
with $a\neq 0$.
Here $h$ is some control function (to be determined) localized in $\omega$, acting either through the KS-KdV equation or elliptic equation; $\gamma_1, \gamma_2>0$ denote the coefficients accounting the long wave instabilities and the short wave dissipation respectively.  Finally, we consider the velocity $c> -\pi^2$, the first eigenvalue of $\partial_{xx}$ with homogeneous Dirichlet boundary conditions, (one can find similar condition on $c$ in \cite{FC13}, \cite{JPPuel14}). 

The  boundary conditions for the state components are given by  
\begin{equation}
	\begin{cases}\label{bd}
		u(t, 0)=0,  \  u(t, 1)=0,  & t \in (0, T), \\
		u_x(t,0)=0, \  u_x(t,1)=0, & t \in (0, T),    \\
		v(t, 0)=0,  \  v(t, 1)=0, & t \in (0, T),
	\end{cases}
\end{equation}
and the initial condition is 
\begin{equation}\label{in}
	u(0, x)=u_0(x), \quad  x \in  (0, 1). 
\end{equation}


Looking into the coupled system \eqref{nonlinear-KS} or \eqref{nonlinear-elliptic}, it is seen that we consider only one localized control function $``h"$ and to obtain the null-controllability for the linearized system, it is necessary to achieve a suitable observability inequality for the associated coupled adjoint system with single observation term (depending on which component we exert the control). To obtain this, we shall use the Carleman estimates for KS and elliptic equations and then remove the  unusual observation terms with help of the coupling term(s). We shall observe this in Sections \ref{Sec-carleman-1} and \eqref{sec-carleman-2}.  Then, applying the source term method followed by a fixed point theorem (see \cite{Tucsnak-nonlinear}), we achieve the local controllability of the nonlinear systems.

\medskip

  We now prescribe the main results of our present work. 
  
\subsubsection{\bf Controllability results of the nonlinear systems}
\begin{theorem}[Control on KS-KdV equation]\label{Thm-nonlinear-KS} 
Let be $(a, b)\in \mathbb R^2$ with $b\neq 0$. Then, the system \eqref{nonlinear-KS}--\eqref{bd}--\eqref{in} is small-time locally null-controllable around the equilibrium $(0,0)$, that is to say, for any given time $T>0$,  there is a $R>0$ such that for chosen initial data $u_0\in H^{-1}(0,1)$ with $\|u_0\|_{H^{-1}(0,1)}\leq R$, there exists a control $h\in L^2((0,T)\times \omega)$  such that the associated solution $(u,v)$ satisfies 
		$$   \left( u(T, x), \, v(T, x) \right) = (0,0), \quad \forall \, x \, \in  (0,1). $$
\end{theorem}
Our next main result is the following: when a control force is acting only on  the elliptic equation. 
\begin{theorem}[Control on elliptic equation]\label{Thm-nonlinear-Ellp}  
Let be $(a, b)\in \mathbb R^2$ with $a\neq 0$. 	Then, the system \eqref{nonlinear-elliptic}--\eqref{bd}--\eqref{in} is small-time locally null-controllable around the equilibrium $(0,0)$, that is, for any given time $T>0$,  there is a $R>0$ such that for chosen initial data $u_0\in H^{-1}(0,1)$ with $\|u_0\|_{H^{-1}(0,1)}\leq R$, there exists a control $h \in L^2((0,T)\times \omega)$  such that the associated solution $(u,v)$ satisfies 
	$$  u(T, x)=0, \quad \forall \, x \, \in  (0,1),  \qquad \limsup_{t\to T^{-}} \|v(t, \cdot)\|_{H^{-1}(0,1)} = 0 . $$
\end{theorem}

\paragraph{\bf Strategy of proofs}
\begin{itemize}
	\item[--] First, we shall prove the global interior null-controllability results of the associated linear models, given by    \begin{equation}
		\begin{cases}
			\label{main}
			u_t+\gamma_1 u_{xxxx}+u_{xxx}+\gamma_2 u_{xx} = a v+\chi_{\omega} h, &\text{in } Q_T,\\
			-v_{xx}+ c v= b u, & \text{in } Q_T, 
		\end{cases} 
	\end{equation}
	or, 
	\begin{equation}
		\begin{cases}
			\label{main-2}
			u_t+\gamma_1 u_{xxxx}+u_{xxx}+\gamma_2 u_{xx} = a v, &\text{in } Q_T,\\
			-v_{xx}+ c v= b u + \chi_{\omega} h, & \text{in } Q_T, 
		\end{cases} 
	\end{equation}
	with the boundary and initial conditions as given by \eqref{bd} and \eqref{in} respectively. 
	
	The global Carleman estimate will help us to prove the null-controllability of the above linearized systems with a proper control cost $C e^{C/T}$. This  is crucial to deduce the local controllability results for the non-linear models. 
	
	\item[--] Next, we shall apply the source term method introduced in \cite{Tucsnak-nonlinear}; more precisely, we prove the null-controllability of our linearized models with additional source term (in the KS equation) from some suitable Hilbert space which is exponentially decreasing while $t\to T^-$.  
	
	\item[--] After that, we use the Banach fixed-point argument to obtain the local null-controllability for the non-linear models.

\end{itemize}

\subsubsection{\bf Controllability results of the associated linearized systems}
%
\begin{theorem}\label{Thm-linear-control-KS}
Let any $u_0\in  H^{-2}(0,1)$ and $T>0$ be given and $(a, b)\in \mathbb R^2$.  Then, we have the following. 
\begin{enumerate} 
\item\label{point-1} When $b\neq 0$,  there exists a control $h\in L^2((0,T)\times \omega)$ such that the system  \eqref{main}--\eqref{bd}--\eqref{in} is null-controllable at time $T$, that is,  
$$\left(u(T,x), \,  v(T,x)\right)=(0,0), \quad \forall \, x  \in  (0,1).$$

\item\label{point-2} Similarly when $a\neq 0$, there exists a control $h\in L^2((0,T)\times \omega)$ such that the solution of \eqref{main-2}--\eqref{bd}--\eqref{in} satisfies 
\begin{align*}
u(T,x)=0, \ \ \forall \, x \in (0,1),  \qquad \limsup_{t\to T^{-}} \|v(t, \cdot)\|_{H^{-2}(0,1)} = 0 .
\end{align*}
\end{enumerate}
 In both cases, the  control function $h$ has the following estimate:
\begin{align}\label{control_estimate}
	\|h\|_{L^2((0,T)\times \omega)} \leq C e^{C/T} \|u_0\|_{H^{-2}(0,1)},
	\end{align}
where the constant $C>0$  neither depends on $T$ nor on $u_0$. 
\end{theorem}

\begin{remark}
	Note that, we need the initial  data $u_0$ to be at least in $H^{-1}(0,1)$ to study the null-controllability of the nonlinear systems whereas to study the linearized models, $u_0\in H^{-2}(0,1)$ is fine. The reason is that we need some more regularity of the set of solutions $(u,v)$ to conclude the required local controllability results for the systems  \eqref{nonlinear-KS} or \eqref{nonlinear-elliptic}; see Section \ref{Section-nonlinear-KS}.
\end{remark}

\smallskip 

\paragraph{\bf Notations}
Throughout the paper, $C$ denotes the generic positive constant which may depend on $\gamma_1, \gamma_2$, $a,b,c$ but does not depend on $u_0$ or $T$, and change from line to line. 

We sometimes use the short notations $H^m(X)$ or $\C^0(X)$ to denote the spaces $H^m(0, T; X)$ and $\C^0([0,T]; X)$ respectively, for some $m\in \zahl$ and $X$ is a Lebesgue space.

\medskip 

\paragraph{\bf Paper organization} The  paper is organized as follows:
\begin{itemize}
	\item [--] In Section \ref{Section-well-posed}, we just point out  the well-posedness of the linearized control systems, which is more or less standard.
	
	\item[--] Section \ref{sec-car1} is	devoted to study the null-controllability of our parabolic-elliptic model where a localized control is acting on the KS-KdV equation. 
	Thereafter, in Section \ref{Section-elliptic}, we study the case  when a control force acts in the elliptic equation only. In both cases, the controllability of the associated linearized models will be established by global Carleman estimates and then using the source term method followed by a fixed point argument, we shall study the nonlinear systems; see Section \ref{Section-nonlinear-KS} for more details.

	 \item[--] Finally, we conclude our paper  with some remarks and open questions related to the KS-KdV-elliptic system, given by Section \ref{Sec-Conclusion}. 
\end{itemize}

\section{Well-posedness of the linearized systems}\label{Section-well-posed}
This section is devoted to discuss the well-posedness of the linear systems \eqref{main}/\eqref{main-2}--\eqref{bd}--\eqref{in}. Let us first write the system \eqref{main} in infinite dimensional ODE setup:
\begin{equation}\label{eq:main abs}
	\begin{cases}
		u_t=\mc Au+ \mc Bh , \quad t\in (0,T),
		\\u(0)=u_0,
	\end{cases}
\end{equation}
where the operators $\mc A$, $\mc B$ are given by 
\begin{equation}
	\begin{cases}\label{eq:operator}\mc Au =-\gamma_1 u_{xxxx}-u_{xxx}- \gamma_2 u_{xx}+ a v,\\
			-v_{xx}+c v= bu, \text{ with } \ v(0)=0,\, v(1)=0,
	\end{cases}
\end{equation} 
and $$D(\mc A)=\big\{u\in H^4(0,1): u(0)=u(1)=u_x(0)=u_x(1)=0\big\}.$$
The control operator $\mc B\in \mathcal{L}\left(\mathbb{R}, L^2(0,1)\right)$ can be defined as 
\begin{equation*} 
	\mc B^*h(t)=\chi_{\omega}h(t, \cdot) . 
\end{equation*}
Next, one can find the adjoint operator $\mc A^*$ of $\mc A$ with  $D(\mc A^{*})=D(\mc A)$ as follows:
\begin{equation}\label{ad}
	\begin{cases}\mc A^*{\sigma}=-\gamma_1 \sigma_{xxxx}+\sigma_{xxx}-\gamma_2 \sigma_{xx}+ b\psi,\\
		-\psi_{xx}+c\psi=a\sigma, \text{ with }  \ \psi(0)=0, \, \psi(1)=0.
	\end{cases}
\end{equation}
Therefore, one has the following adjoint system to \eqref{main}/\eqref{main-2}--\eqref{bd}--\eqref{in}, given by
 \begin{equation}\label{Adjoint-system}
 	\begin{cases}
 		-\sigma_t + \gamma_1 \sigma_{xxxx} - \sigma_{xxx}  +  \gamma_2 \sigma_{xx}=  b \psi   + g, & (t,x) \in Q_T, \\
 		-\psi_{xx}+c\psi=a \sigma  , & (t,x) \in Q_T,\\
 		\sigma(t, 0)=0,\quad \sigma(t, 1)=0,  & t \in (0, T),\\
 		\sigma_x(t,0)=0,\quad \sigma_x(t,1)=0, & t \in (0, T),\\
 		\psi(t, 0)=0,\quad \psi(t, 1)=0, & t \in (0, T),\\
 		\sigma(T,x)=\sigma_T(x), & x \in (0,1),
 	\end{cases}
 \end{equation}
with given final data $\sigma_T \in H^2_0(0,1)$ and right hand sides $g\in L^2(0,T; L^2(0,1))$. 

With these, we have the following proposition. 
\begin{proposition}\label{Prop-energy}
	For given 
 $\sigma_T \in H^2_0(0,1)$ and $g\in L^2(0,T; L^2(0,1))$ or $g\in L^1(0,T; H^2_0(0,1))$, there exists unique solution $(\sigma, \psi)$ to \eqref{Adjoint-system} such that 
\begin{align*}
&	\sigma \in \C^0([0,T]; H^2_0(0,1)) \cap L^2(0,T; H^4(0,1)) \cap H^1(0,T; L^2(0,1)) , \\
	& \psi \in L^2(0,T; H^2(0,1)\cap H^1_0(0,1)), 
\end{align*}
and in addition,  they satisfy the following estimates 
\begin{subequations}
	\begin{align}
	\label{adj-solu-1}	
	&\|\sigma\|_{\C^0(H^2_0) \cap L^2(H^4) \cap H^1(L^2)  } \leq C e^{C T} \left( \|\sigma_T\|_{H^2_0(0,1)} + \|g\|_{L^2(L^2)}    \right) , \\
	\label{adj-solu-2}	
	&\|\psi\|_{L^2(H^2 \cap H^1_0)} \leq C e^{C T} \left( \|\sigma_T\|_{H^2_0(0,1)} + \|g\|_{L^2(L^2)}    \right),
	\end{align}
\end{subequations}
or, 
\begin{subequations}
	\begin{align}
		\label{adj-solu-1-1}	
		&\|\sigma\|_{\C^0(H^2_0) \cap L^2(H^4) \cap H^1(L^2)  } \leq  C e^{C T} \left( \|\sigma_T\|_{H^2_0(0,1)} + \|g\|_{L^1(H^2_0)}    \right) , \\
		\label{adj-solu-2-1}	
		&\|\psi\|_{L^2(H^2 \cap H^1_0)} \leq   C e^{C T} \left( \|\sigma_T\|_{H^2_0(0,1)} + \|g\|_{L^1(H^2_0)}    \right) , 
	\end{align}
\end{subequations}
where $C>0$ is some constant that may depend on $a, b, c, \gamma_1, \gamma_2$ but not on $u_0$ or $T$.  
\end{proposition}
\begin{proof}

	We start the proof with the estimate of the elliptic equation. 
	Multiplying the second equation of \eqref{Adjoint-system} by $\psi$, then integrating by parts and considering  the homogeneous boundary conditions on $\psi$,  we obtain
	\begin{align*}
		\int_{0}^{1}\psi_x^2+c\int_{0}^{1}\psi^2=a\int_{0}^{1}\sigma \psi.
	\end{align*}
Recall the Poincar\'{e} inequality with best constant:
\begin{equation}\label{poin} 
	\int_{0}^{1}|\phi|^2\leq \frac{1}{\pi^2}\int_{0}^{1}|\phi'|^2, \quad \forall \phi \in H^1_0(0,1).
\end{equation}
Then, using the Cauchy-Schwartz and Young's  inequality together with \eqref{poin},  we  deduce
	\begin{align*}
		\norm{\psi(t, \cdot)}_{H^1_0(0,1)}&\leq C \norm{\sigma(t, \cdot)}_{L^2(0,1)}.
	\end{align*}

	Next, multiplying the second equation of \eqref{Adjoint-system} by $\psi_{xx}$ and performing the same business as above, we get the following estimate:
	\begin{align*}\norm{\psi(t, \cdot)}_{H^2(0,1)\cap H^1_0(0,1)}^2 \leq C \norm{\sigma(t, \cdot)}_{L^2(0,1)}^2, \text{ for a.e. t in } [0,T].
	\end{align*}
	Then, integrating on $[0,T]$, we have
	\begin{equation*}
		\norm{\psi}_{L^2(0,T;H^2(0,1)\cap H^1_0(0,1))}^2 \leq C \norm{\sigma}_{L^2(0,T; L^2(0,1))}^2.
	\end{equation*}

\smallskip 

To obtain the regularity result for $\sigma$, let us make a change of variable from $t$ to $T-t$ so that the first equation of \eqref{Adjoint-system} becomes forward in time. Then, testing the  equation against $\sigma_{xxxx}$, we get 
	\begin{align}\label{equ-mult}
		\int_0^1  \sigma_t \sigma_{xxxx} + \gamma_1 \int_0^1 |\sigma_{xxxx}|^2 = \int_0^1 \left( b\psi + \sigma_{xxx} -\gamma_2 \sigma_{xx} + g\right) \sigma_{xxxx} .	 
	\end{align}
	Performing consecutive  integration by parts  in   the first integral and using the Young's inequality in the right hand side of  \eqref{equ-mult}, leads to
	\begin{align*}
		\frac{d}{d t} \int_0^1 |\sigma_{xx}|^2 + \gamma_1 \int_0^1 |\sigma_{xxxx}|^2 \leq \epsilon  \int_0^1 |\sigma_{xxxx}|^2 + \frac{C}{\epsilon} \int_0^1 \left(|\psi|^2 + |\sigma_{xxx}|^2 + |\sigma_{xx}|^2 + |g|^2 \right) .
	\end{align*}
	Using the estimate of $\psi$ from \eqref{adj-solu-2} and choosing $\epsilon=\gamma_1/2$, we get 
	\begin{align}\label{aux-energy}
	\frac{d}{d t} \int_0^1 |\sigma_{xx}|^2 + \frac{\gamma_1}{2} \int_0^1 |\sigma_{xxxx}|^2 \leq  C \int_0^1 \left(|\sigma|^2 + |\sigma_{xxx}|^2 + |\sigma_{xx}|^2 + |g|^2 \right) .
		\end{align} 
	Thanks to the Ehrling’s lemma, we have for $\epsilon>0$, 
	\begin{align*}
		\int_0^1 |\sigma_{xxx}|^2 \leq \epsilon 	\int_0^1 |\sigma_{xxxx}|^2 + K(\epsilon)\int_0^1 |\sigma|^2,
	\end{align*}
and applying this in \eqref{aux-energy} for  $\epsilon>0$ small enough and the Poincar\'e inequality,   we have  
\begin{align}\label{aux-energy-2}
	\frac{d}{d t} \int_0^1 |\sigma_{xx}|^2 +  \int_0^1 |\sigma_{xxxx}|^2 \leq  C \int_0^1 \left(|\sigma_{xx}|^2 + |g|^2 \right) .
\end{align} 
Then the Gronwall's lemma yields to
\begin{align*}
	\|\sigma\|_{L^{\infty}(0,T; H^2_0(0,1) )} \leq  C e^{C T} \left(\|\sigma_T\|_{H^2_0(0,1)} + \|g\|_{L^2(0,T; L^2(0,1))} \right) .
\end{align*}   
Next, integrating both sides of the inequality \eqref{aux-energy-2} with respect to $(0,t)$ and taking supremum on the both sides for $t\in [0,T]$, we  shall obtain the $L^2(0,T; H^4(0,1))$ estimate of $\sigma$ and finally from the equation of $\sigma$, we obtain $\sigma_t\in L^2(0,T; L^2(0,1))$.  The estimate of $\sigma_t$ can be obtained using the $L^\infty(H^2_0)$ and $L^2(H^4)$ estimations of $\sigma$. This ensures that $\sigma \in \C^0([0,T]; H^2_0(0,1)).$

\medskip

On the other hand, assuming $g\in L^1(0, T; H^2_0(0,1))$, one can obtain  the  estimates \eqref{adj-solu-1-1} and \eqref{adj-solu-2-1}.  
\end{proof}


Next let us give the well-posedness of the linearized system \eqref{main}--\eqref{bd}--\eqref{in} and \eqref{main-2}--\eqref{bd}--\eqref{in}. 
We only write the details  for the system \eqref{main}--\eqref{bd}--\eqref{in}. For a second order parabolic-elliptic control problem, similar well-posedness results have been appeared in \cite{hernandez2021boundary}.

\begin{definition}
	Let be $u_0\in H^{-2}(0,1)$ and $h\in L^2(0,T; H^{-2}(0,1))$. Then,  $(u,v)\in [L^2(0,T; L^2(0,1))]^2$ is said to be a solution to the system \eqref{main}--\eqref{bd}--\eqref{in} if for any $g\in L^2(0,T; L^2(0,1))$ the following  holds:
	\begin{align}\label{defn of soln}
		\Iint u(t,x) g(t,x)=&\ip{u_0}{\sigma(0, x)}_{H^{-2}(0,1), H^2_0(0,1)}
		+\ip{h}{\sigma}_{L^2(0,T; H^{-2}(0,1)), L^2(0,T; H_{0}^{2}(0,1))},
	\end{align}
where $(\sigma, \psi)$ is a solution of the adjoint system \eqref{Adjoint-system} with $\sigma_T=0$. 
\end{definition}

Moreover, we have the following result. 
\begin{theorem}
Let $u_0\in H^{-2}(0,1)$ and $h\in L^2(0,T; H^{-2}(0,1)).$ Then the system \eqref{main}--\eqref{bd}--\eqref{in} has a unique solution $(u,v) \in  [\C^0([0,T]; H^{-2}(0,1))]^2 \cap [L^2(0,T; L^2(0,1))]^2$ with the following estimate:
\begin{equation}\label{uv}
	\norm{(u,v)}_{[\C^0(0,T; H^{-2}(0,1))]^2}  + 	\norm{(u,v)}_{[L^2(0,T; L^2(0,1))]^2}   \leq C e^{C T} \left(\norm{u_0}_{H^{-2}(0,1)}+ \norm{h}_{L^2(0,T; H^{-2}(0,1))}\right).
\end{equation}
\end{theorem}
\begin{proof}
	Let us define a linear map $\mathcal{L}: L^2(0,T; L^2(0,1))\mapsto \mathbb{R}$ by 
	\begin{equation*}
	\mathcal{L}(g)= \ip{u_0}{\sigma(0, x)}_{H^{-2}(0,1), H^2_0(0,1)}+\ip{h}{\sigma}_{L^2(0,T; H^{-2}(0,1)), L^2(0,T; H_{0}^{2}(0,1))},
	\end{equation*}
where $(\sigma,\psi)$ is the solution of the equation \eqref{Adjoint-system}. Thanks to \Cref{Prop-energy}, $\mathcal{L}$ is a continuous linear functional. Thus, by the Riesz-representation theorem, we have a unique $(u,v)\in [L^2(0,T; L^2(0,1))]^2$ such that \eqref{defn of soln} holds.

\smallskip 
Next, if we define the map $\mathcal{L}:L^1(0, T; H^2_0(0,1)) \mapsto \mathbb{R}$ by \eqref{l}, then using estimate \eqref{adj-solu-1-1}, we get the continuity of $\mathcal{L}$. This yields that $u\in L^{\infty}(0, T; H^{-2}(0,1))$ satisfying \eqref{defn of soln} and consequently, $v \in L^\infty(0,T; H^{-2}(0,1))$. 
Indeed, we have the following estimate
\begin{equation}\label{l}
	\norm{(u,v)}_{L^{\infty}(0,T; H^{-2}(0,1)\times H^{-2}(0,1) )}\leq  C e^{C T}\left(\norm{u_0}_{H^{-2}(0,1)}+ \norm{h}_{L^2(0,T; H^{-2}(0,1))}\right).
\end{equation}
Using standard density arguments, we get the estimate \eqref{uv}. Indeed, first taking the initial data regular enough and using the equation \eqref{main}, it can be shown that $u_t\in L^2(0,T; H^{-4}(0,1)).$ Using $u\in L^2(0,T;L^2(0,1))$, by classical regularity result, one can conclude that $u\in \C^0([0,T]; H^{-2}(0,1)),$ and hence 
$v\in \C^0([0,T]; H^{-2}(0,1))$. For the less regular initial data, standard density argument and the estimate \eqref{l} will give the desired result. Similar arguments can be found in \cite[Theorem 2.4]{Cerpa-Mercado-Pazoto}.
\end{proof}

Analogous to the previous theorem,  one can prove the following result when a  control is acting in the elliptic equation.
\begin{theorem}
	Let $u_0\in H^{-2}(0,1)$ and $h\in L^2(0,T; H^{-2}(0,1)).$ Then the system \eqref{main-2}--\eqref{bd}--\eqref{in} has a unique solution $(u,v)$ such that $u\in  \C^0([0,T]; H^{-2}(0,1))\cap L^2(0,T; L^2(0,1))$, $v  \in L^2(0,T; L^2(0,1))$, and in addition, one has 
	\begin{align*}
		\norm{u}_{\C^0([0,T];H^{-2}(0,1)) \cap L^2(0,T; L^2(0,1))}  + 	\norm{v}_{L^2(0,T; L^2(0,1))}   \leq C e^{C T} \left(\norm{u_0}_{H^{-2}(0,1)}+ \norm{h}_{L^2(0,T; H^{-2}(0,1))}\right).
	\end{align*}
\end{theorem}

	\smallskip 
	
\section{Control acting in the KS-KdV equation}\label{sec-car1}

This section is devoted to study the controllability of our coupled system when  a localized control is exerted in the KS-KdV equation. We start by proving the null-controllability of the associated linearized model and as mentioned earlier, using the source term method and a fixed point argument, we will give the required local null-controllability for the nonlinear system. 

\subsection{A global Carleman estimate}\label{Sec-carleman-1}

 We shall utilize the Carleman technique to deduce the null-controllability of the linearized system. 
 
To do so, let us define the some useful weight functions.
Consider a non-empty open set $\omega_0\subset \subset  \omega$.
There exists a function $\nu \in \C^\infty([0,1])$ such that  
\begin{align}\label{definition_nu} 
	\begin{cases}
		\nu(x)>0 \qquad \forall x\in (0,1), \ \ \nu(0)=\nu(1)=0, \\
		|\nu^\prime(x) | \geq \widehat c >0 \ \  \forall x \in \overline{(0,1) \setminus \omega_0} \ \ \text{ for some } \widehat c > 0.
	\end{cases}
\end{align} 
It is clear that $\nu^{\prime}(0)>0$ and $\nu^\prime(1)<0$. We refer \cite{Fur-Ima} where the existence of such function has been addressed.  

Now, for any constant $k>1$ and parameter $\lambda>0$,  we define the weight functions 
\begin{align}\label{weight_function}
	\vphi(t,x)= \frac{e^{2\lambda k \|\nu\|_{\infty}   }-  e^{\lambda\big( k\|\nu\|_{\infty} + \nu(x) \big)}}{t(T-t)}	, \quad \xi(t,x) = \frac{e^{\lambda\big( k\|\nu\|_{\infty} + \nu(x) \big)}}{t(T-t)}, \quad \forall \, (t,x) \in Q_T.
\end{align}
We have defined the above  weight functions by means of  the works \cite{Zhou} and \cite{G07}. It is clear that both $\vphi$ and $\xi$ are positive functions in $Q_T$.

We also define
\begin{equation}
\begin{aligned}\label{min max}
&\widehat{\vphi}(t)=\max_{x\in[0,1]}\vphi(t,x), \quad \quad {\vphi}^*(t)=\min_{x\in[0,1]}\vphi(t,x),\\
&\widehat{\xi}(t)=\max_{x\in[0,1]}\xi(t,x), \quad \quad {\xi}^*(t)=\min_{x\in[0,1]}\xi(t,x).
\end{aligned}
\end{equation}  
Some immediate results associated with the weight functions are the followings.
\begin{itemize}
	\item[--]  For any $l\in \mathbb N^*$ and $s>0$, there exists some $C>0$ such that
	\begin{align}\label{weight-deri-t}
		\big|(e^{-2s\vphi}\xi^l)_t \big| \leq C s e^{-2s\vphi}\xi^{l+2}.
	\end{align}

	\item[--]  For any $(l,n)\in \mathbb N^* \times \mathbb N^*$ and $s>0$, there exists some $C>0$ such that
\begin{align}\label{weight-deri-x}
	\big|(e^{-2s\vphi}\xi^l)_{n,x} \big| \leq C s^n \lambda^n e^{-2s\vphi}\xi^{l+n}.
\end{align}  
	
\end{itemize}

\paragraph{\bf Some useful notations} We also  declare the following notations which will simplify the expressions of our Carleman inequalities.

\begin{itemize} 
	\item[--]  For any $\sigma\in \C^2([0,T];\C^4([0,1]))$ and positive parameters $s$, $\lambda$,  we denote
	\begin{multline}\label{Notation_KS}
		I_{KS}(s,\lambda;\sigma):=s^7\lambda^8 \iintQ e^{-2s\vphi} \xi^7 |\sigma|^2 + s^5\lambda^6 \iintQ e^{-2s\vphi} \xi^5 |\sigma_{x}|^2 +  s^3\lambda^4 \iintQ e^{-2s\vphi} \xi^3 |\sigma_{xx}|^2 \\
		+ s\lambda^2 \iintQ e^{-2s\vphi} \xi |\sigma_{xxx}|^2   +	s^{-1} \iintQ e^{-2s\vphi}\xi^{-1} (|\sigma_t|^2+ |\sigma_{xxxx}|^2\big). 
	\end{multline}	
	
	\item[--] For any function $\psi\in \C^2(\overline{Q_T})$ and positive parameters $s$, $\lambda$, we denote 
	\begin{align}\label{Notation_elliptic}
		I_E(s,\lambda; \psi): = s^3\lambda^4 \iintQ e^{-2s\vphi} \xi^3 |\psi|^2 + s\lambda^2 \iintQ e^{-2s\vphi} \xi |\psi_{x}|^2 +	s^{-1} \iintQ e^{-2s\vphi}\xi^{-1} |\psi_{xx}|^2.
	\end{align}
\end{itemize}



\smallskip 
With help of the above notations, we now prescribe the  Carleman estimate satisfied by the solution $(\sigma, \psi)$ to the following adjoint system
(of \eqref{main}/\eqref{main-2}--\eqref{bd}--\eqref{in})
\begin{equation}\label{Adjoint-system-2}
	\begin{cases}
		-\sigma_t + \gamma_1 \sigma_{xxxx} - \sigma_{xxx}  +  \gamma_2 \sigma_{xx}=  b \psi   , & (t,x) \in Q_T, \\
		-\psi_{xx}+c\psi=a \sigma   , & (t,x) \in Q_T,\\
		\sigma(t, 0)=0,\quad \sigma(t, 1)=0,  & t \in (0, T),\\
		\sigma_x(t,0)=0,\quad \sigma_x(t,1)=0, & t \in (0, T),\\
		\psi(t, 0)=0,\quad \psi(t, 1)=0, & t \in (0, T),\\
		\sigma(T,x)=\sigma_T(x), & x \in (0,1),
	\end{cases}
\end{equation}
with given final data $\sigma_T \in H^2_0(0,1)$. 
\begin{theorem}[Carleman estimate: control in KS-KdV eq.]\label{Thm_Carleman_main}
		Let the weight functions $(\vphi, \xi)$  be given by \eqref{weight_function}. Then, there exist positive constants $\lambda_0$, $s_0:=  \mu_0(T+T^{2})$ for some $\mu_0>0$ and $C$ which  depend on $\gamma_1$, $\gamma_2$, $a$, $b$, $c$ and the set $\omega$,  such that we have the following estimate satisfied by the solution to \eqref{Adjoint-system-2}, 
			\begin{align}\label{carle-joint} 
			I_{KS}(s,\lambda; \sigma) +   	I_{E}(s,\lambda; \psi)
			\leq C  s^{15}\lambda^{16} \int_0^T \int_{\omega_1} e^{-4s\vphi^*+2s\widehat{\vphi}} (\widehat{\xi})^{15} |\sigma|^2,
		\end{align}
		for all $\lambda \geq \lambda_0$ and $s\geq s_0$, where $I_{KS}(s,\lambda; \sigma)$ and $I_E(s,\lambda; \psi)$ are introduced in \eqref{Notation_KS} and \eqref{Notation_elliptic} respectively and $\vphi^*, \widehat{\vphi},\, \widehat{\xi}$ are given in \eqref{min max}. 
	\end{theorem}
Before going to the proof of the above Carleman inequality, let us write the individual Carleman inequalities for the components $\sigma$ and $\psi$.
 \begin{proposition}\label{thm-carleman-sigma}
	Let the weight functions $(\vphi, \xi)$  be given by \eqref{weight_function}. Then, there exist positive constants $\lambda_1$, $s_1:=  \mu_1(T+T^{2})$ for some $\mu_1>0$ and $C$ which  depend on $\gamma_1, \gamma_2, a,b,c$ and the set $\omega_0$,  such that we have the following estimate satisfied by the solution component $\sigma$  of \eqref{Adjoint-system-2},
	\begin{align}\label{carle-ks}
		I_{KS}(s,\lambda; \sigma)   
		\leq C \iintQ e^{-2s\vphi}|\psi|^2 + Cs^7\lambda^8 \int_0^T\int_{\omega_0} e^{-2s\vphi} \xi^7 |\sigma|^2,
	\end{align}
	for all $\lambda \geq \lambda_1$ and $s\geq s_1$, and $I_{KS}(s,\lambda; \sigma)$ is introduced by \eqref{Notation_KS}.
\end{proposition}
The proof of  above such Carleman estimate is initially established in \cite{Zhou}. We also refer \cite{Cerpa-Mercado-Pazoto} where a similar Carleman estimate has been addressed.  

Let us also write the following elliptic Carleman estimate (see for instance \cite{Fur-Ima}) satisfied by $\psi$. 
 \begin{proposition}\label{thm-carleman-psi}
	Let the weight functions $(\vphi, \xi)$  be given by \eqref{weight_function}. Then, there exist positive constants $\lambda_2$, $s_2:=  \mu_2(T+T^{2})$ with some $\mu_2>0$ and $C$ which  depend on $\gamma_1, \gamma_2, a,b,c$ and the set $\omega_0$,  such that we have the following estimate satisfied by the solution component $\psi$  of \eqref{Adjoint-system-2}, 
	\begin{align}\label{carle-ell} 
		I_{E}(s,\lambda; \psi)
		\leq C \iintQ e^{-2s\vphi}|\sigma|^2 + Cs^3\lambda^4 \int_0^T\int_{\omega_0} e^{-2s\vphi} \xi^3 |\psi|^2,
	\end{align}
	for all $\lambda \geq \lambda_2$ and $s\geq s_2$,  where $I_{E}(s,\lambda; \psi)$ is introduced by \eqref{Notation_elliptic}.
\end{proposition}
With help of the above two Carleman estimates \eqref{carle-ks} and \eqref{carle-ell}, we are now ready to prove the main Carleman estimate \eqref{carle-joint}.

\begin{proof}[\bf Proof of \Cref{Thm_Carleman_main}]
	Let us add both the Carleman estimates \eqref{carle-ks} and \eqref{carle-ell}, we have 
	\begin{multline}\label{Add_carlemans}
		I_{KS}(s,\lambda; \sigma) + I_{E}(s,\lambda; \psi) 
	\leq C \bigg[\iintQ e^{-2s\vphi}|\psi|^2 + 	\iintQ e^{-2s\vphi}|\sigma|^2  \\+
	 s^7\lambda^8 \int_0^T\int_{\omega_0} e^{-2s\vphi} \xi^7 |\sigma|^2 
	 + s^3\lambda^4 \int_0^T\int_{\omega_0} e^{-2s\vphi} \xi^3 |\psi|^2\bigg],
	\end{multline}
for all $\lambda\geq \lambda_0:= \max\{\lambda_1, \lambda_2\}$ and $s\geq \widehat\mu(T+T^2)$ with $\widehat \mu:=\max\{\mu_1, \mu_2\}$.

\smallskip 
\noindent 
{\bf Step 1. Absorbing the lower order integrals.} Observe that $1\leq C T^{2}\xi$ for some $C>0$. This yields 
\begin{align}\label{lower_int}
	 \iintQ e^{-2s\vphi}|\psi|^2 + 	\iintQ e^{-2s\vphi}|\sigma|^2 \leq C T^{6}\iintQ e^{-2s\vphi}\xi^3|\psi|^2 +  C T^{14} \iintQ e^{-2s\vphi}\xi^7|\sigma|^2. 
	\end{align}
Then by choosing $s \geq \widehat C T^2$ for some constant $\widehat C>0$,   the integrals in the right hand side of \eqref{lower_int} can be absorbed by the associated leading integrals appearing in the left hand side of \eqref{Add_carlemans}. This leads to the following 
\begin{align}\label{Add_carlemans-2}
	I_{KS}(s,\lambda; \sigma) + I_{E}(s,\lambda; \psi) 
	\leq C \bigg[
	s^7\lambda^8 \int_0^T\int_{\omega_0} e^{-2s\vphi} \xi^7 |\sigma|^2 
	+ s^3\lambda^4 \int_0^T\int_{\omega_0} e^{-2s\vphi} \xi^3 |\psi|^2\bigg],
\end{align}
for all $\lambda\geq \lambda_0$ and $s\geq \mu_0(T+T^2)$ for some $\mu_0\geq \widehat \mu$. 

\smallskip 
\noindent 
 {\bf Step 2. Absorbing the observation integral in $\psi$.}
Consider $\omega_0$ and another non-empty open set $\omega_1$ in such a way that  $\omega_0 \subset \subset \omega_1 \subset \subset \omega$. Then, we consider the function 
\begin{align}\label{smooth_func}
	\phi \in \C^\infty_c(\omega_1)  \ \text{ with } \ 0\leq \phi \leq 1, \ \text{ and } \ \phi =1 \ \text{ in } \ \omega_0. 
\end{align}
Now, recall the adjoint system \eqref{Adjoint-system-2},  one has (since $b\neq 0$),
\begin{align*}
	\psi = \frac{1}{b} \left( -\sigma_t + \gamma_1 \sigma_{xxxx} -\sigma_{xxx} +\gamma_2 \sigma_{xx}   \right), \quad \text{in } Q_T.
\end{align*}
Using this, we have 
\begin{align}\label{Obs-psi}
s^3\lambda^4 \int_0^T\int_{\omega_0} e^{-2s\vphi} \xi^3 |\psi|^2 
&\leq s^3\lambda^4 \int_0^T\int_{\omega_1} \phi e^{-2s\vphi} \xi^3 |\psi|^2   \notag \\ 
&= \frac{1}{b}s^3\lambda^4 \int_0^T\int_{\omega_1} \phi e^{-2s\vphi} \xi^3 \psi \left( -\sigma_t + \gamma_1 \sigma_{xxxx} -\sigma_{xxx} +\gamma_2 \sigma_{xx}   \right)  \notag \\  
 &=A_1+ A_2+ A_3+A_4 .
\end{align}

\noindent
(i) Let us start with the following. 
\begin{align}\label{A_1}
	A_1:&= - \frac{1}{b}s^3\lambda^4 \int_0^T\int_{\omega_1} \phi e^{-2s\vphi} \xi^3 \psi \sigma_t  \notag  \\
	&= \frac{1}{b} s^3\lambda^4 \int_0^T\int_{\omega_1} \phi (e^{-2s\vphi} \xi^3)_t \psi \sigma
	 + \frac{1}{b}s^3\lambda^4 \int_0^T\int_{\omega_1} \phi e^{-2s\vphi} \xi^3 \psi_t \sigma.
	\end{align}
In above, we perform an integration by parts with respect to $t$. There is no boundary terms since   the weight function $e^{-2s\vphi}$ vanishes near $t=0$ and $T$.  

Now, recall the fact \eqref{weight-deri-t}, so that the first integral in the right hand side of \eqref{A_1} gives 
\begin{align}\label{esti_A_1-1}
	\left|\frac{1}{b} s^3\lambda^4 \int_0^T\int_{\omega_1} \phi (e^{-2s\vphi} \xi^3)_t \psi \sigma\right| &\leq C s^4\lambda^4 \int_0^T \int_{\omega_1} e^{-2s\vphi} \xi^5 \left| \psi \sigma \right| \notag \\ 
	&\leq \epsilon s^3\lambda^4 \iintQ e^{-2s\vphi} \xi^3 |\psi|^2 + \frac{C}{\epsilon} s^5\lambda^4\int_0^T \int_{\omega_1} e^{-2s\vphi} \xi^7 |\sigma|^2,
\end{align}
for any $\epsilon>0$, where we have  used the Cauchy-Schwarz and Young's inequalities. 

To estimate the last integral term in \eqref{A_1}, we need to differentiate the second equation of \eqref{Adjoint-system-2} with respect to $t$ and that yields
\begin{equation}\label{elliptic_t}
	\begin{cases}
		-(\psi_t)_{xx}+ c \psi_t=a \sigma_t, & \text{ in }Q_T,\\
		\psi_t(t,0)= \psi_t(t,1)=0, & t \in (0,T).
	\end{cases}
\end{equation}
	Thanks to \eqref{elliptic_t}, there exists a constant $C>0$ such that we have
\begin{align}\label{lap est}
	\int_{0}^{1}|\psi_t|^2\leq C \int_{0}^{1}|\sigma_t|^2.
\end{align}
Let us estimate the last integral in \eqref{A_1} as 
\begin{align}\label{esti_A_1-2}
\nonumber	\left|\frac{1}{b} s^3\lambda^4 \int_0^T\int_{\omega_1} \phi e^{-2s\vphi} \xi^3 \psi_t \sigma\right|&\leq C s^3\lambda^4 \int_0^T\int_{\omega_1} e^{-2s\vphi} \xi^3 \left|\psi_t\right| \, |\sigma|\\ 
\nonumber& \leq C s^3\lambda^4 \int_0^T e^{-2s\vphi^*} (\widehat{\xi})^3 \norm{\psi_t}_{L^2(\omega_1)}\, \norm{\sigma}_{L^2(\omega_1)}, \text{ using } \eqref{min max}\\
 \nonumber& \leq C s^3\lambda^4 \int_0^T e^{-s\widehat{\vphi}} (\widehat{\xi})^3 \norm{\sigma_t}_{L^2(\omega_1)}\, e^{-2s\vphi^*+s\widehat{\vphi}} \norm{\sigma}_{L^2(\omega_1)} \text{ using } \eqref{lap est}\\
 \nonumber& \leq C  \int_0^T e^{-s\widehat{\vphi}}  (s\widehat{\xi})^{-\frac{1}{2}}  \norm{\sigma_t}_{L^2(\omega_1)}\, e^{-2s\vphi^*+s\widehat{\vphi}} \lambda^4 (s\widehat{\xi})^{\frac{7}{2}} \norm{\sigma}_{L^2(\omega_1)} \\
	&\leq  \epsilon s^{-1} \iintQ e^{-2s\widehat{\vphi}} (\widehat{\xi})^{-1} |\sigma_t|^2 + \frac{C}{\epsilon} s^7\lambda^8\int_0^T \int_{\omega_1} e^{-4s\vphi^*+2s\widehat{\vphi}} (\widehat{\xi})^7 |\sigma|^2,
\end{align}
for any $\epsilon>0$, where we have  used the Cauchy-Schwarz and Young's inequalities.
Thus, we eventually have 
\begin{align}\label{A_1-1}
	|A_1| \leq \epsilon s^3\lambda^4 \iintQ e^{-2s\vphi} \xi^3 |\psi|^2 +   \epsilon s^{-1} \iintQ e^{-2s\vphi} \xi^{-1} |\sigma_t|^2 + \frac{C}{\epsilon} s^7\lambda^8 \int_0^T \int_{\omega_1} e^{-4s\vphi^*+2s\widehat{\vphi}} (\widehat{\xi})^7 |\sigma|^2 . 
 \end{align}

\noindent
(ii) Next, after integrating by parts twice with respect to $x$, we get 
\begin{align}\label{A_2}
	A_2:&=  \frac{1}{b} \gamma_1s^3\lambda^4 \int_0^T\int_{\omega_1} \phi e^{-2s\vphi} \xi^3 \psi \sigma_{xxxx}  \notag  \\
	&= -\frac{1}{b}\gamma_1 s^3\lambda^4 \bigg[\int_0^T\int_{\omega_1} \phi_x (e^{-2s\vphi} \xi^3) \psi \sigma_{xxx} 
	+\int_0^T\int_{\omega_1} \phi (e^{-2s\vphi} \xi^3)_x \psi \sigma_{xxx}   \notag  \\
	& \qquad \qquad \qquad \qquad \qquad \qquad \qquad +\int_0^T\int_{\omega_1} \phi (e^{-2s\vphi} \xi^3) \psi_x \sigma_{xxx}\bigg]  \notag \\
	& = \frac{1}{b}\gamma_1 s^3\lambda^4 \int_0^T \int_{\omega_1} \bigg[\phi_{xx} (e^{-2s\vphi} \xi^3) \psi \sigma_{xx} +  \phi (e^{-2s\vphi} \xi^3)_{xx} \psi \sigma_{xx} +  \phi (e^{-2s\vphi} \xi^3) \psi_{xx} \sigma_{xx} \bigg]    \notag \\  
	& \quad + \frac{2}{b}\gamma_1 s^3\lambda^4 \int_0^T \int_{\omega_1} \bigg[\phi_{x} (e^{-2s\vphi} \xi^3)_x \psi \sigma_{xx} + \phi_{x} (e^{-2s\vphi} \xi^3) \psi_x \sigma_{xx} + \phi (e^{-2s\vphi} \xi^3)_x \psi_x \sigma_{xx}\bigg].
\end{align}
Using \eqref{weight-deri-x} and Cauchy-Schwarz inequality,    we have from \eqref{A_2},
\begin{align}\label{A_2-1}
	|A_2| \leq 3\epsilon s^3\lambda^4 \iintQ e^{-2s\vphi} \xi^{3} |\psi|^2+2\epsilon s\lambda^2 \iintQ e^{-2s\vphi} \xi |\psi_x|^2 + \epsilon s^{-1} \iintQ e^{-2s\vphi} \xi^{-1} |\psi_{xx}|^2  \notag \\
+	 \frac{C}{\epsilon} s^{7}\lambda^{8} \int_0^T \int_{\omega_1}  e^{-2s\vphi} \xi^{7} |\sigma_{xx}|^2 .
\end{align}

\noindent
(iii) Let us estimate the third term of \eqref{Obs-psi} in the following way 
\begin{align*}
	A_3:&=  -\frac{1}{b}s^3\lambda^4\int_0^T\int_{\omega_1} \phi e^{-2s\vphi} \xi^3 \psi \sigma_{xxx}  \\
	&= \frac{1}{b} s^3\lambda^4 \int_0^T\int_{\omega_1} \bigg[\phi_x e^{-2s\vphi} \xi^3 \psi \sigma_{xx}  +  \phi (e^{-2s\vphi} \xi^3)_x \psi \sigma_{xx} +  \phi e^{-2s\vphi} \xi^3 \psi_x \sigma_{xx}\bigg].
\end{align*}
Therefore, for $\epsilon>0$ we have, using \eqref{weight-deri-x} and the Cauchy-Schwarz inequality, that 
\begin{align}\label{A_3} 
	|A_3| \leq 2\epsilon s^3\lambda^4 \iintQ e^{-2s\vphi} \xi^3 |\psi|^2 + 
	\epsilon s\lambda^2 \iintQ e^{-2s\vphi} \xi |\psi_x|^2 + 
	\frac{C}{\epsilon} s^5\lambda^6 \int_0^T \int_{\omega_1} e^{-2s\vphi} \xi^5 |\sigma_{xx}|^2.  
\end{align}

\noindent
(iv) Next, it is easy to see that 
\begin{align}\label{A_4}
	|A_4|  \leq \epsilon s^3\lambda^4 \iintQ e^{-2s\vphi} \xi^3 |\psi|^2 + \frac{C}{\epsilon} s^3\lambda^4 \int_0^T \int_{\omega_1} e^{-2s\vphi} \xi^3 |\sigma_{xx}|^2. 
\end{align}
Combining the estimates \eqref{A_1-1}, \eqref{A_2-1}, \eqref{A_3}, \eqref{A_4}, and applying in \eqref{Obs-psi} we get 
\begin{multline}\label{estimate-aux-1}
	s^3\lambda^4 \int_0^T \int_{\omega_0} e^{-2s\vphi} \xi^3 |\psi|^2 
	\leq 
7\epsilon  s^3\lambda^4 \iintQ  e^{-2s\vphi} \xi^3 |\psi|^2 +  
	3\epsilon s\lambda^2 \iintQ  e^{-2s\vphi} \xi |\psi_x|^2 \\ 
		+ \epsilon s^{-1} \iintQ  e^{-2s\vphi} \xi^{-1} |\psi_{xx}|^2+\epsilon s^{-1} \iintQ  e^{-2s\vphi} \xi^{-1} |\sigma_{t}|^2 \\
		+\frac{C}{\epsilon} s^7\lambda^8 \int_0^T \int_{\omega_1} e^{-4s\vphi^*+2s\widehat{\vphi}} (\widehat{\xi})^7 |\sigma|^2 + \frac{C}{\epsilon} s^7\lambda^8 \int_0^T \int_{\omega_1} e^{-2s\vphi} \xi^7 |\sigma_{xx}|^2.
\end{multline} 
Now, fix $\epsilon>0$  small enough in \eqref{estimate-aux-1}, so that all the terms with coefficient $\epsilon$ can be absorbed by the associated integrals in the left hand side of \eqref{Add_carlemans-2}. As a  consequence,  the estimate \eqref{Add_carlemans-2} boils  down to 
\begin{align}\label{Add_carlemans-3}
	I_{KS}(s,\lambda; \sigma) + I_{E}(s,\lambda; \psi) 
	\leq C s^7\lambda^8 \int_0^T \int_{\omega_1} e^{-4s\vphi^*+2s\widehat{\vphi}} (\widehat{\xi})^7 |\sigma|^2 
	+C s^7\lambda^8 \int_0^T\int_{\omega_1} e^{-2s\vphi} \xi^7 |\sigma_{xx}|^2,
\end{align}
for all $\lambda\geq \lambda_0$ and $s\geq \mu_0(T+T^2)$.

\smallskip 
\noindent 
{\bf Step 3. Absorbing the observation integral in $\sigma_{xx}$.}
We need to estimate the last integral of \eqref{Add_carlemans-3}. Consider a function (recall that $\omega_1\subset\subset \omega$)
\begin{align}\label{smooth_func-2}
	\wphi \in \C^\infty_c(\omega)  \ \text{ with } \ 0\leq \wphi \leq 1, \ \text{ and } \  \wphi =1 \ \text{ in } \ \omega_1. 
\end{align}
With this, we have 
\begin{align*}  
	s^{7}\lambda^{8} \int_0^T \int_{\omega_1}  e^{-2s\vphi} \xi^{7} |\sigma_{xx}|^2 \leq s^{7}\lambda^{8} \int_0^T \int_{\omega} \wphi e^{-2s\vphi} \xi^{7} |\sigma_{xx}|^2 .
\end{align*}
Successive integrating by parts yields to
\begin{align*}
	&s^{7}\lambda^{8} \int_0^T \int_{\omega} \wphi e^{-2s\vphi} \xi^{7} \sigma_{xx} \sigma_{xx}  \\
	= & -	s^{7}\lambda^{8} \int_0^T \int_{\omega} \bigg[ \wphi e^{-2s\vphi} \xi^{7} \sigma_{xxx} \sigma_{x}   + \wphi (e^{-2s\vphi} \xi^{7})_x \sigma_{xx} \sigma_{x} + \wphi_x e^{-2s\vphi} \xi^{7} \sigma_{xx} \sigma_{x} \bigg] \\
	= &   s^{7}\lambda^{8} \int_0^T \int_{\omega} \bigg[ \wphi e^{-2s\vphi} \xi^{7} \sigma_{xxxx} \sigma   + \wphi (e^{-2s\vphi} \xi^{7})_{xx} \sigma_{xx} \sigma + \wphi_{xx} e^{-2s\vphi} \xi^{7} \sigma_{xx} \sigma \bigg] \\
	& \qquad +  2s^{7}\lambda^{8} \int_0^T \int_{\omega} \bigg[ \wphi_x e^{-2s\vphi} \xi^{7} \sigma_{xxx} \sigma   + \wphi (e^{-2s\vphi} \xi^{7})_{x} \sigma_{xxx} \sigma + \wphi_{x} (e^{-2s\vphi} \xi^{7})_x \sigma_{xx} \sigma \bigg].
\end{align*}
Again by using the information \eqref{weight-deri-x} and Cauchy-Schwarz inequality,   we get for some $\epsilon>0$, that 
\begin{multline} \label{A_2-2}
	s^{7}\lambda^{8} \int_0^T \int_{\omega_1}  e^{-2s\vphi} \xi^{7} |\sigma_{xx}|^2   
	\leq  
	\epsilon s^{-1} \iintQ e^{-2s\vphi} \xi^{-1} |\sigma_{xxxx}|^2 + 2 \epsilon s\lambda^2 \iintQ e^{-2s\vphi} \xi |\sigma_{xxx}|^2 \\
	+ 
	3 \epsilon s^3\lambda^4 \iintQ e^{-2s\vphi} \xi^3 |\sigma_{xx}|^2  +  \frac{C}{\epsilon} s^{15}\lambda^{16} \int_0^T \int_{\omega} e^{-2s\vphi} \xi^{15} |\sigma|^2.
\end{multline}
We fix small enough $\epsilon>0$ in \eqref{A_2-2} so that the integrals with coefficient $\epsilon$ can be absorbed by the leading integrals in the left hand of \eqref{Add_carlemans-3} and as a result we obtain 
\begin{align*}
	I_{KS}(s,\lambda; \sigma) + I_{E}(s,\lambda; \psi) 
	\leq C 
	s^{15}\lambda^{16} \int_0^T \int_{\omega_1} e^{-4s\vphi^*+2s\widehat{\vphi}} (\widehat{\xi})^{15} |\sigma|^2 ,
\end{align*}
for all $\lambda\geq \lambda_0$ and $s\geq \mu_0(T+T^2)$.

This is the required joint Carleman estimate \eqref{carle-joint} of our theorem. The proof is finished.
\end{proof}

\subsection{Observability inequality and null-controllability of the linearized model}\label{Section-obser}
With the  Carleman estimate \eqref{carle-joint},  it is difficult to obtain the desired observability inequality with norm of the  components $\sigma, \psi$ at $x=0$ in the left hand side. The reason behind is that the weight function $e^{-2s\vphi}$  is vanishing near $t=0$.  

To avoid this obstacle, we  shall proceed as follows. 
For given $T>0$, there is some $N\in \mathbb N^*$, such that $\frac{T}{N}\leq 1$. We then consider
\begin{align}\label{function_lm} 
	\ell_N(t)=
	\begin{cases}
		\frac{T^2(2N-1)}{4N^2} , &  0\leq t\leq  \frac{T}{2N}, \\
		t(T-t)  , &  \frac{T}{2N} \leq t\leq T,
	\end{cases}
\end{align}
and the following modified weight functions
\begin{align}\label{weight_function_not_0}
	\mathfrak{S}_N(t,x)= \frac{e^{2 \lambda k \|\nu\|_{\infty} } -  e^{\lambda\big( k\|\nu\|_{\infty} + \nu(x) \big)}}{\ell_N (t)}	, \quad \mathfrak{Z}_N(t,x) = \frac{e^{\lambda\big( k\|\nu\|_{\infty} + \nu(x) \big)}}{\ell_N(t)}, \quad \forall (t,x) \in Q_T.
\end{align}
for any constants $\lambda>1$ and $k>1$. Set $\widehat{\mathfrak{S}}_N, {\mathfrak{S}}^*_N, \widehat{\mathfrak{Z}}_N, {\mathfrak{Z}}^*_N$ in a similar fashion as \eqref{min max}.
%
Further, we denote
\begin{align}
\label{M-hat}	\widehat M :  &=  e^{2\lambda k\|\nu\|_{\infty}} - e^{\lambda k\|\nu\|_{\infty}}, \\
\label{M-star}	M^* :  &= e^{2\lambda k\|\nu\|_{\infty}} - e^{\lambda(k+1)\|\nu\|_{\infty}},
\end{align}
Thus we note that, 
\begin{equation}
	2M^*-\widehat{M}>0 \text{ and also } 2\mathfrak{S}^*_N-\widehat{\mathfrak{S}}_N>0.
\end{equation}
\begin{remark}\label{remark_N}
	Observe that when $0<T\leq 1$, one can simply take $N=1$ in the definition of $\ell_N(t)$  in \eqref{function_lm}.
\end{remark}

\begin{proposition}[Observability inequality: control in KS-KdV eq.]\label{prop:refined_a_0}
	There exists a constant $C>0$  depending on $\omega$, $a$, $b$, $c$, $\gamma_1$, $\gamma_2$ and $N$  such that
	we have the following observability inequality
	\begin{align}\label{Partial-obser-1}
		\|\sigma(0,\cdot)\|^2_{H^2_0(0,1)} + 	\|\psi(0,\cdot)\|^2_{H^2(0,1)}\leq
	C e^{C/T} \int_0^T\int_{\omega} |\sigma|^2  ,
	\end{align}
	where $(\sigma,\psi)$ is the solution to the adjoint system \eqref{Adjoint-system-2} and the constant $C>0$ is independent in $T$. 
\end{proposition}

\begin{proof}

	By  construction we have that $\vphi=\mathfrak{S}_N$ and $\xi=\mathfrak{Z}_N$ in $\big(\frac{T}{2N},T\big)\times(0,1)$, hence
	{
	\begin{align*}
		\int_{\frac{T}{2N}}^T \int_0^1 e^{-2s\mathfrak{S}_N} \mathfrak{Z}^7_N |\sigma|^2 = 	\int_{\frac{T}{2N}}^T \int_0^1 e^{-2s\vphi} \xi^7 |\sigma|^2.
	\end{align*}
}
	Therefore,  from the Carleman estimate \eqref{carle-joint} we readily get
	{ 
	\begin{align}\label{carleman-T/2}
		s^7\lambda^8 \int_{\frac{T}{2N}}^T \int_0^1 e^{-2s\mathfrak{S}_N} \mathfrak{Z}^7_N |\sigma|^2 
		\leq Cs^{15}\lambda^{16} \int_0^T \int_{\omega_1} e^{-4s\mathfrak{S}^*_N+2s\widehat{\mathfrak{S}}_N} (\widehat{\mathfrak{Z}}_N)^{15} |\sigma|^2   ,
	\end{align}
}
for all $\lambda \geq \lambda_0$ and $s\geq s_0$ and the constant $C>0$ does not depend on $T$.

{ Let us introduce a function $\eta\in \C^1([0,T])$ such that
	\begin{equation*}
		\eta=1 \textnormal{ in } \left[0,\frac{T}{2N}\right], \quad \eta=0 \textnormal{ in } \left[\frac{3T}{4N}, T  \right].
	\end{equation*}
It is clear that $\textnormal{Supp} \, (\eta^\prime) \subset \big[\frac{T}{2N}, \frac{3T}{4N}\big]$.} 

Recall the adjoint system  \eqref{Adjoint-system-2}. Then, the  pair $(\widetilde \sigma, \widetilde \psi)$ with  $\widetilde \sigma=\eta \sigma$, $\widetilde \psi=\eta \psi$ satisfies the following set of equations
	\begin{align*}
		\begin{cases}
			-\widetilde \sigma_t + \gamma_1 \widetilde \sigma_{xxxx} - \widetilde \sigma_{xxx}  + \gamma_2 \widetilde \sigma_{xx} = b \widetilde \psi    -\eta{^\prime} \sigma,    &\text{in } \big(0,\frac{3T}{4N}\big)\times (0,1)  , \\
			-\widetilde \psi_{xx} + c \widetilde \psi = a \widetilde \sigma ,  &\text{in } \big(0,\frac{3T}{4N}\big)\times (0,1), \\
			\widetilde \sigma(t,0) = 	\widetilde \sigma(t,1) = 0, &\text{in } \big(0,\frac{3T}{4N}\big), \\ 
			\widetilde \sigma_x (t,0)   =	\widetilde \sigma_x (t,1)=0,  &\text{in } \big(0,\frac{3T}{4N}\big), \\
			\widetilde \psi(t,0) =\widetilde \psi(t,1) =0,   &\text{in } \big(0,\frac{3T}{4N}\big), \\
			\widetilde \sigma(T) = 0, &\text{in } (0,1).
		\end{cases}
	\end{align*}
Thanks to the \Cref{Prop-energy}, we have the following energy estimate for $\widetilde \sigma$,
{
	\begin{align}\label{eq:est_eta_forward}
		\|\eta \sigma\|_{L^\infty(0,\frac{3T}{4N};H_0^2(0,1))} 
		&	\leq 
			C e^{\frac{3CT}{4N}} \|\eta^\prime \sigma\|_{L^2( (0,\frac{3T}{4N})\times (0,1) ) } \notag \\ 
			&  \leq  
C e^{\frac{3C}{4}} \|\sigma\|_{L^2((\frac{T}{2N}, \frac{3T}{4N})  \times (0,1))},
	\end{align}
where we have used the fact that $\frac{T}{N}\leq 1$.} Moreover, the constant $C>0$ in \eqref{eq:est_eta_forward}  does not depend on $T$. 

From \eqref{eq:est_eta_forward}, we further  observe that 
\begin{align}\label{subsequent-1} 
	\|\eta \sigma\|^2_{L^\infty(0,\frac{3T}{4N};H_0^2(0,1))}
	& \leq  C  s^7\lambda^8\int_{\frac{T}{2N}}^{\frac{3T}{4N}} \int_0^1 e^{2s\mathfrak{S}_N} \mathfrak{Z}^{-7}_N e^{-2s\mathfrak{S}_N} \mathfrak{Z}^7_N|\sigma|^2  \notag \\   
	&  \leq  C \max_{\big[\frac{T}{2N}, \frac{3T}{4N}\big]\times [0,1]}  \left|e^{2s\mathfrak{S}_N} \mathfrak{Z}^{-7}_N\right|
	s^7\lambda^8 \int_{\frac{T}{2N}}^{T} \int_0^1 e^{-2s\mathfrak{S}_N} \mathfrak{Z}^7_N|\sigma|^2   \notag \\		 
	&\leq C \max_{\big[\frac{T}{2N}, \frac{3T}{4N}\big]\times [0,1]}  \left|e^{2s\mathfrak{S}_N} \mathfrak{Z}^{-7}_N\right| s^{15}\lambda^{16}  \int_0^T \int_{\omega_1} e^{-4s\mathfrak{S}^*_N+2s\widehat{\mathfrak{S}}_N} (\widehat{\mathfrak{Z}}_N)^{15} |\sigma|^2 ,  
\end{align}
where we have utilized the Carleman estimate  \eqref{carleman-T/2} to obtain the last inequality in \eqref{subsequent-1}.

\smallskip 

Let us now use the following: { (i) when $T\leq 1$, we take $N=1$ (see Remark \ref{remark_N}) and the maximum of $e^{2s\mathfrak{S}_1} {\mathfrak{Z}}^{-7}_1$ in  $\big[\frac{T}{2}, \frac{3T}{4}\big]$ occurs at  $t=\frac{3T}{4}$;   (ii) when $T > 1$, then of course we need to take $N\geq 2$ and  the maximum of the function $e^{2s\mathfrak{S}_N} {\mathfrak{Z}}^{-7}_N$ in  $\big[\frac{T}{2N}, \frac{3T}{4N}\big]$ occurs at $\frac{T}{2N}$. To be more precise,  we have 
\begin{align}\label{maximum}
	\begin{cases}
\text{when $T\leq 1$,} \ \ \max\limits_{\big[\frac{T}{2}, \frac{3T}{4}\big]\times [0,1]}  \left|e^{2s\mathfrak{S}_1} \mathfrak{Z}^{-7}_1\right| \leq C T^{14} e^{\frac{32s\widehat M}{3T^2}},
\\
\text{when $T> 1$,} \ \ \max\limits_{\big[\frac{T}{2N}, \frac{3T}{4N}\big]\times [0,1]}  \left|e^{2s\mathfrak{S}_N} \mathfrak{Z}^{-7}_N\right| \leq C T^{14} e^{\frac{8s\widehat M}{T^2}\frac{N^2}{(2N-1)}},
\end{cases}
\end{align}
with some constant $C>0$ that does depend on $N$ and the  quantities $\omega$, $a$, $b$, $c$, $\gamma_1$, $\gamma_2$ but not on T.}

We also observe that the maximum of $e^{-4s\mathfrak{S}^*_N+2s\widehat{\mathfrak{S}}_N} (\widehat{\mathfrak{Z}}_N)^{15}$ in $(0,T)$ occurs at $t=T/2$, and one has 
{
\begin{align}\label{maximum-2}
\max_{[0,T]\times [0,1]} 	\big|e^{-4s\mathfrak{S}^*_N+2s\widehat{\mathfrak{S}}_N} (\widehat{\mathfrak{Z}}_N)^{15}\big|  \leq \frac{C}{T^{30}} \, e^{-\frac{8s(2M^*-\widehat M)}{T^2}} .
\end{align} 
}
{ Combining \eqref{maximum}  and \eqref{maximum-2}, we have 
\begin{align}\label{maximum-3}
	\begin{cases}
	\text{when $T\leq 1$,} \ \ \max\limits_{\big[\frac{T}{2}, \frac{3T}{4}\big]\times [0,1]}  \left|e^{2s\mathfrak{S}_1} \mathfrak{Z}^{-7}_1\right| \times \max\limits_{[0,T]\times [0,1]} \big|	e^{-4s\mathfrak{S}^*_1+2s\widehat{\mathfrak{S}}_1} (\widehat{\mathfrak{Z}}_1)^{15} \big| 
	\leq \frac{C}{T^{16}} \, e^{\frac{8s}{T^2}\big[ \frac{7}{3}\widehat M - 2M^* \big]  } , 
	\\
	\text{when $T> 1$,} \ \ \max\limits_{\big[\frac{T}{2N}, \frac{3T}{4N}\big]\times [0,1]}  \left|e^{2s\mathfrak{S}_N} \mathfrak{Z}^{-7}_N\right| \times \max\limits_{[0,T]\times [0,1]} 	\big|e^{-4s\mathfrak{S}^*_N+2s\widehat{\mathfrak{S}}_N} (\widehat{\mathfrak{Z}}_N)^{15}\big| 
	\\ 
\qquad \qquad \qquad \qquad \qquad \qquad \qquad 
\qquad \qquad \qquad \qquad \qquad \qquad \quad \leq 
	  \frac{C}{T^{16}}\, e^{\frac{8s}{T^2}\big[(\frac{N^2+2N-1}{2N-1}) \widehat M - 2M^*\big]},
	\end{cases}
	\end{align}
}
where $\widehat M$ and $M^*$ are defined in \eqref{M-hat}--\eqref{M-star}.

 \smallskip

 {
Hereinafter, we fix $\lambda=\lambda_0$, $s= s_0=\mu_0(T+T^2)$ and using \eqref{maximum-3} in \eqref{subsequent-1}, we obtain 
\begin{align}\label{subsequent-2}
	\|\sigma(0,\cdot) \|^2_{H^2_0(0,1))} 
	&	\leq  \frac{C}{T^{16}}  e^{\frac{C s_0}{T^2}} \int_0^T \int_{\omega}   |\sigma|^2    \notag \\
	&\leq C e^{C/T} \int_0^T \int_{\omega}   |\sigma|^2 ,
\end{align}
 for some constant $C>0$  that may differ from previous but independent in $T$.}

Next, from the equation of $\widetilde \psi$, we get 
\begin{align*}
	\|\eta \psi(t, \cdot)\|_{H^2(0,1)} \leq |a| \| \eta \sigma \|_{L^\infty(0,\frac{3T}{4N}; L^2(0,1))} .
\end{align*} 
 This, combining with  the inequality \eqref{subsequent-2}, we deduce the required observability estimate \eqref{Partial-obser-1}.
  \end{proof}

Below, we sketch the proof of null-controllability for the corresponding linearized system when a control acts on the KS-KdV equation.  
 \begin{proof}[\bf Proof of \Cref{Thm-linear-control-KS}--\Cref{point-1}] 
  	Once we have the  observability inequality \eqref{Partial-obser-1}, 
  	then by standard duality argument, one can prove the existence of a null-control $h\in L^2((0,T)\times \omega)$ (in terms of the adjoint component $\sigma$ localized in $\omega$) 
  	for the linearized system \eqref{main}--\eqref{bd}--\eqref{in}. 
  	This method is standard and has been applied in several works; see for instance the pioneer works \cite{Fur-Ima} or \cite[Section 1.3]{Cara-Guerrero}.
  	
  {	Finally, we mention that the cost of control: $\displaystyle C e^{C/T}\|u_0\|_{H^{-2}(0,1)}$ is followed from the sharp observability inequality \eqref{Partial-obser-1}, where $C>0$ is independent in $T$. } 
  \end{proof}

\smallskip
  
\subsection{Null controllability of the nonlinear system: control in KS-KdV equation}\label{Section-nonlinear-KS}
In this section, by using the control cost   of the linear system, given by 
\Cref{Thm-linear-control-KS}--\Cref{point-1}, we shall prove the local null-controllability of our nonlinear model \eqref{nonlinear-KS}--\eqref{bd}--\eqref{in} for given initial data  $u_0\in H^{-1}(0,1)$. The proof will be based on  the  so-called source term method developed in \cite{Tucsnak-nonlinear} followed by a Banach fixed point argument. 
\subsubsection{\bf The source term method}  
 Let us discuss the source term method (see \cite{Tucsnak-nonlinear}) for our case.
We assume the constants  $p>0$, $q>1$  in such a way that 
\begin{align}\label{choice-p_q}
	1<q<\sqrt{2}, \ \ \text{and} \ \ p> \frac{q^2}{2-q^2}.
\end{align}
We also redenote  the constant appearing in the control estimate \eqref{control_estimate} of the linearized model by $K$, more precisely the control cost is given by {$\displaystyle Ke^{K/T}$} (to make difference with the generic constant $C$). Now define the functions 
\begin{align}\label{def_weight_func}
	\begin{cases}
		\rho_0(t) = e^{-\frac{pK}{(q-1)(T-t)}}, \\
		\rho_{\F}(t)= e^{-\frac{(1+p)q^2 K}{(q-1)(T-t)}}, 
	\end{cases}
	\qquad \forall t \in \left[ T\left(1-\frac{1}{q^2}  \right), T\right],
\end{align} 
extended in $\left[0, T(1-1/q^2) \right]$ in a constant way  such that the functions $\rho_0$ and $\rho_{\F}$ are continuous and non-increasing in $[0,T]$ with $\rho_0(T)=\rho_{\F}(T)=0$. 
\begin{remark}\label{Remark-weights}
	We compute that
	\begin{align*}
		\frac{\rho_0^2(t)}{\rho_\F(t)}= e^{\frac{q^2 K + pK(q^2-2)}{(q-1)(T-t)}}, \quad \forall t \in \left[T\left(1-\frac{1}{q^2}\right) , T\right],
	\end{align*} 
	Due to the choices of $p,q$ in \eqref{choice-p_q}, we have $K\big(q^2+ p(q^2-2)\big)<0$, $(q-1)>0$ and therefore we can conclude that
	\begin{align*}
		\frac{\rho^2_0(t)}{\rho_{\F}(t)} \leq 1, \quad \forall t \in [0,T]. 
	\end{align*}
\end{remark}

Let us write the following Gelfand tripel  
\begin{align}
	H^3(0,1)\cap H^2_0(0,1)\subset H^{-1}(0,1)\subset  \left(H^3(0,1)\cap H^2_0(0,1)\right)^\prime , 
\end{align} 
and denote $X:= [L^2(0,1)]^2$. Then, we define the following weighted spaces: 
\begin{subequations}
	\begin{align}
		\label{space_F}
		&\F:= \left\{f\in L^2(0,T; (H^{3}(0,1)\cap H^2_0(0,1))') \ \Big| \ \frac{f}{\rho_\F} \in L^2(0,T; (H^{3}(0,1)\cap H^2_0(0,1))')  \right\}, \\
		\label{space_Y} 
	&	\Y := \left\{ (u,v)\in L^2(0,T; X)  \ \Big| \ \frac{(u,v)}{\rho_0} \in L^2(0,T; X)  \right\}, \\
		\label{space_V}
	&	\V:= \left\{ h\in L^2((0,T) \times \omega)  \ \Big| \ \frac{h}{\rho_0}\in L^2((0,T) \times \omega)  \right\}.
	\end{align}
\end{subequations}
It is clear that the aforementioned spaces are  Hilbert spaces. In particular,
the inner products in  $\F$ and $\V$ are respectively defined by 
\begin{align*}
	\langle f_1, f_2 \rangle_{\F} := \int_0^T  \rho^{-2}_\F \langle f_1, f_2 \rangle_{(H^{3}(0,1)\cap H^2_0(0,1))'}, \quad 
		\langle h_1, h_2 \rangle_{\V} := \int_0^T\int_\omega  \rho^{-2}_0 h_1 h_2,
\end{align*}
 and consequently, the norms by
 \begin{align*}
 	\|f\|_{\F} := \left(\int_0^T  \rho^{-2}_\F \|f\|^2_{(H^{3}(0,1)\cap H^2_0(0,1))'}\right)^{1/2}, \quad 
 	\|h\|_{\V} := \left(\int_0^T\int_\omega  \rho^{-2}_0 |h|^2\right)^{1/2}.
 \end{align*}

Let us now consider the following system (still linear) with a source term $f$ in the right hand side of the KS-KdV equation:
 \begin{equation}\label{System-source}
 	\begin{cases}
 		u_t+\gamma_1 u_{xxxx}+u_{xxx}+\gamma_2 u_{xx} = f+ a v+\chi_{\omega} h, &\text{in } Q_T,\\
 		-v_{xx}+ c v= b u, & \text{in } Q_T,\\ 
 		u(t, 0)=0,  \  u(t, 1)=0,  & t \in (0, T),\\
 		u_x(t,0)=0, \  u_x(t,1)=0, & t \in (0, T),\\
 		v(t, 0)=0,  \  v(t, 1)=0, & t \in (0, T),\\
 		u(0, x)=u_0(x), & x \in  (0, 1).
 	\end{cases}
 \end{equation}

The classical parabolic regularity gives the following result (see \cite[Chapter 3]{DLv5} or \cite[Section 3, Chapter II]{RT}  for more details).
\begin{proposition}\label{prop-esti-regular}
	For any given initial data $u_0 \in H^{-1}(0,1)$, source term $f \in L^2(0,T; (H^{3}\cap H^2_0)^\prime)$ and $h\in L^2((0,T)\times \omega)$,  there exists unique solution $(u,v)$ of \eqref{System-source} such that
	\begin{align}\label{sol-u}
		& u \in \C^0([0,T]; H^{-1}(0,1)) \cap L^2(0,T; H^1_0(0,1)) \cap H^1(0,T;(H^{3}(0,1)\cap H^2_0(0,1))^\prime), \\
		\label{sol-v}
		& v \in \C^0([0,T]; H^{-1}(0,1)) \cap L^2(0,T; H^2(0,1)\cap H^1_0(0,1)).
	\end{align} 
	In addition, we have the following estimate
	\begin{multline*}
		\left\|{u} \right\|_{\C^0(H^{-1}) \cap L^2(H^1_0) \cap H^1((H^{3}\cap H^2_0)^\prime)} + 	\left\|v \right\|_{\C^0( H^{-1}) \cap L^2(H^2\cap H^1_0)} \\
		\leq  Ce^{CT} \left(\|u_0\|_{H^{-1}(0,1)} + \left\|f \right\|_{L^2((H^{3}\cap H^2_0)^\prime)} + \|h\|_{L^2((0,T)\times \omega  )}\right),
	\end{multline*}
for some constant $C>0$  independent in  $T$. 
\end{proposition} 

Using the above proposition and the embedding 
\begin{align}\label{embedding}
	\Big\{u\in L^2(0,T; H^1_0(0,1)) \mid u_t \in L^2(0,T; (H^{3}(0,1)\cap H^2_0(0,1))') \Big\} \hookrightarrow L^4(0,T; L^2(0,1)), 
	\end{align}
we deduce   that  $u\in L^4(0,T; L^2(0,1))$.

Next, by the following proposition we shall obtain the existence of a control $h\in \V$  for the system \eqref{System-source} with  given source term $f\in \F$ and initial data $u_0 \in H^{-1}(0,1)$.  Observe that, we choose here slightly higher regular initial data than the linear system (see \Cref{Thm-linear-control-KS}). The reason is to handle the nonlinear term $uu_x$ in our system \eqref{nonlinear-KS}--\eqref{bd}--\eqref{in}.

Let us  prove the following result. 
\begin{proposition}\label{Proposition-weighted}
	Let $T>0$,  $u_0 \in H^{-1}(0,1)$  and $f\in \F$ be given. We also take any $(a,b)\in \mathbb R^2$ with $b\neq 0$. Then, there exists a linear map 
	\begin{align}\label{map-linearity}
		\left[u_0, f\right] \in H^{-1}(0,1) \times \F \mapsto \left[(u,v), h  \right] \in \Y \times \V,
	\end{align}
	such that $\left[(u,v), h  \right]$ solves \eqref{System-source}. Moreover, we have the following regularity estimate
	\begin{multline}\label{estimate-weighted}
		\left\|\frac{u}{\rho_0} \right\|_{\C^0(H^{-1}) \cap L^2(H^1_0) \cap H^1((H^{3}\cap H^2_0)')} + 	\left\|\frac{v}{\rho_0} \right\|_{\C^0(H^{-1}) \cap L^2(H^2\cap H^1_0)} + \left\|\frac{h}{\rho_0} \right\|_{L^2((0,T)\times \omega)}\\
		\leq   C_1  \left(\|u_0\|_{H^{-1}(0,1)} + \left\|\frac{f}{\rho_\F} \right\|_{L^2((H^{3}\cap H^2_0)')} \right),
	\end{multline}
where the constant $C_1>0$ does not depend on $u_0$, $f$, $h$. 
\end{proposition}
\begin{proof}
	Let any $T>0$ be given. Then, we define the sequence $(T_k)_{k\geq 0}$ as 
	\begin{align}\label{def-T-k}
		T_k = T - \frac{T}{q^k}, \qquad \forall k\geq 0,
		\end{align}
	where $q$ is given by \eqref{choice-p_q}. With this $T_k$, it can be shown that the weight functions  $\rho_0$ and $\rho_\F$ enjoy the following relation:
\begin{align}\label{relation_rho_0_rho_F}
	\rho_0(T_{k+2}) =  \rho_{\F}(T_k) e^{\frac{K}{T_{k+2}-T_{k+1}}} , \qquad \forall k\geq 0.
\end{align}
	We also define  a sequence $(m_k)_{k\geq 0}$ with
\begin{align}\label{sequence_m_k}
	m_0 = u_0 \in H^{-1}(0,1), \qquad m_{k+1} =  \widetilde u( T^{-}_{k+1}), \quad \forall k \geq 0 ,
\end{align}
where $\left(\widetilde u, \widetilde v\right)$ is such that 
\begin{align*}
 &\widetilde u \in \C^0([T_k, T_{k+1}]; H^{-1}(0,1)) \cap L^2(T_k,T_{k+1}; H^1_0(0,1)) \cap H^1(T_k,T_{k+1}; (H^{3}(0,1)\cap H^2_0(0,1))'), \\
& \widetilde v \in \C^0([T_k, T_{k+1}]; H^{-1}(0,1)) \cap L^2(T_k, T_{k+1}; H^2(0,1)\cap H^1_0(0,1)), 
 \end{align*}
and that $(\widetilde u, \widetilde v)$ uniquely satisfies the following set of equations 
 \begin{equation}\label{System-source-no-control}
	\begin{cases}
		\widetilde u_t+\gamma_1\widetilde u_{xxxx}+ \widetilde u_{xxx}+\gamma_2 \widetilde u_{xx} = f+ a \widetilde v &\text{in }    (T_k, T_{k+1})\times (0,1),\\
		-\widetilde v_{xx}+ c\widetilde v= b \widetilde u, & \text{in }  (T_k, T_{k+1})\times (0,1),\\ 
	\widetilde	u(t, 0)=0,  \  \widetilde u(t, 1)=0,  & t \in  (T_k, T_{k+1}),\\
	\widetilde	u_x(t,0)=0, \ \widetilde u_x(t,1)=0 & t \in (T_k, T_{k+1}) \\
	\widetilde	v(t, 0)=0,  \  \widetilde v(t, 1)=0 & t \in  (T_k, T_{k+1}),\\
		\widetilde u(T^+_{k}, x) = 0, & x \in  (0, 1).
	\end{cases}
\end{equation}
 Moreover, it satisfies the following regularity estimate (using \Cref{prop-esti-regular}) 
 \begin{align*}
 	\left\|{\widetilde u} \right\|_{\C^0([T_k, T_{k+1}]; H^{-1}) \cap L^2(T_k, T_{k+1}; H^1_0) \cap H^1(T_k, T_{k+1}; (H^{3}\cap H^2_0)')} + 	\left\|\widetilde v \right\|_{\C^0([T_k, T_{k+1}]; H^{-1}) \cap L^2(T_k, T_{k+1}; H^2\cap H^1_0)}\\ 
 	\leq   Ce^{CT}  \left\|f \right\|_{L^2(T_k, T_{k+1}; (H^{3}\cap H^2_0)^\prime)} , 
 \end{align*}
for all $k\geq 0$. In particular,  we have 
\begin{align}\label{esti-m-k}
	\left\|m_{k+1} \right\|_{H^{-1}(0,1)} \leq   Ce^{CT} \left\|f \right\|_{L^2(T_k, T_{k+1}; (H^{3}\cap H^2_0)')} , \quad \forall k\geq 0.
\end{align} 

\noindent 
$\bullet$ {\em Weighted estimate of the control.}
 Let us  consider the following control system in the time interval $(T_k, T_{k+1})$ for all $k\geq 0$,
 \begin{equation}\label{eq tk}
 	\begin{cases}
 		\bar u_t+\gamma_1\bar u_{xxxx}+ \bar u_{xxx}+\gamma_2 \bar u_{xx} = \chi_{\omega} h_k+ a \bar v &\text{in }    (T_k, T_{k+1})\times (0,1),\\
 		-\bar v_{xx}+ c\bar v= b \bar u, & \text{in }  (T_k, T_{k+1})\times (0,1),\\ 
 		\bar	u(t, 0)=0,  \  \bar u(t, 1)=0,  & t \in  (T_k, T_{k+1}),\\
 		\bar	u_x(t,0)=0, \ \bar u_x(t,1)=0 & t \in (T_k, T_{k+1}) \\
 		\bar	v(t, 0)=0,  \  \bar v(t, 1)=0 & t \in  (T_k, T_{k+1}),\\
 		\bar u(T^+_{k}, x) = m_k, & x \in  (0, 1).
 	\end{cases}
 \end{equation}
By using \Cref{Thm-linear-control-KS}--\Cref{point-1}, we have the existence of a control $h_k$ that satisfies 
\begin{align}\label{control_estimate tk}
	\|h_k\|_{L^2((T_k, T_{k+1})\times \omega)} \leq K e^{K/\left( T_{k+1}-T_k\right)} \|m_k\|_{H^{-1}(0,1)}, \quad \forall k\geq 0, 
\end{align}
where the constant $C>0$ (appearing in \Cref{Thm-linear-control-KS}) neither depends on $T$ nor on $u_0$. As a result,  the solution to  \eqref{eq tk} satisfies
\begin{align}\label{controllability-kth}
	\left(\bar u(T^-_{k+1}, x), \bar v(T^-_{k+1}, x)\right)=(0, 0), \quad \forall x \in (0,1), \quad \forall k\geq 0. 
\end{align}
Now, combining \eqref{esti-m-k} and \eqref{control_estimate tk}, 
  we have for all $k\geq 0$, 
 \begin{align}\label{hk1}
 	\|h_{k+1}\|_{L^2((T_{k+1},T_{k+2})\times \omega)} &\leq K  e^{K/\left( T_{k+2}-T_{k+1}\right)} \|m_{k+1}\|_{H^{-1}(0,1)}    \notag \\
 	&\leq {C e^{CT} e^{K/\left( T_{k+2}-T_{k+1}\right)} }\left\|f \right\|_{L^2(T_k, T_{k+1}; (H^{3}\cap H^2_0)')} \notag \\
 	& \leq C e^{CT} {e^{K/(T_{k+2} - T_{k+1})}} \rho_{\F}(T_k) \left\|\frac{f}{\rho_{\F}} \right\|_{L^2(T_k, T_{k+1}; (H^{3}\cap H^2_0)')},
 \end{align}
since $\rho_\F$ is a non-increasing function. Then, using the relation \eqref{relation_rho_0_rho_F}, one can write 
\begin{equation}
	\norm{{h}_{k+1}}_{L^2((T_{k+1},T_{k+2})\times \omega)}\leq C e^{CT}\rho_0(T_{k+2})\norm{\frac{f}{\rho_{\F}}}_{L^2(T_k,T_{k+1};(H^{3}\cap H^2_0)')}, \quad \forall k \geq 0.  
\end{equation}
 Again, since $\rho_0$ is also non-increasing, we deduce that 
\begin{align}\label{esti-contr-k}
	\norm{\frac{{h}_{k+1}}{\rho_0}}_{L^2((T_{k+1},T_{k+2})\times \omega)} & \leq\frac{1}{\rho_0(T_{k+2})}\norm{{h}_{k+1}}_{L^2((T_{k+1},T_{k+2})\times \omega)}  \notag \\
	&\leq C e^{CT} \norm{\frac{f}{\rho_{\F}}}_{L^2(T_k,T_{k+1};(H^{3}\cap H^2_0)')}, \text{ for all } k \geq 0.
\end{align}
We now define the control function ${h}$ as
\begin{equation}
	h:=\sum_{k\geq 0} {h}_k \chi_{(T_k,T_{k+1})},   \quad \text{in } \ (0,T)\times \omega, 
\end{equation}
where $\chi$ denotes the characteristic function. 

Recall that, we have the $L^2$- estimates of $\displaystyle \frac{{h}_{k+1}}{\rho_0}$ for all $k\geq0$. It remains to find the $L^2$- estimate of $\displaystyle \frac{{h}_0}{\rho_0}$.
 From the control  estimate \eqref{control_estimate tk} and since $\displaystyle \rho_0(T_1) = e^{-\frac{pqK}{T(q-1)}}$ and $\displaystyle T_1=T\left(1-\frac{1}{q}\right)$ (from \eqref{def-T-k}), we get
\begin{align}\label{esti-h_0}
	\norm{{h}_0}_{L^2((0,T_1)\times \omega)} & \leq {Ke^{\frac{K}{T_1}} }\norm{u_0}_{H^{-1}(0,1)}  \notag \\
	& =  {Ke^{ \frac{1}{T} \frac{q(1+p)K}{(q-1)}} \rho_0(T_1)} \|u_0\|_{H^{-1}(0,1)} \notag  \\
	& \leq {K e^{\frac{K_1}{T}}} \rho_0(T_1) \|u_0\|_{H^{-1}(0,1)}, 
\end{align}
where $K_1 := \frac{q(1+p)K}{(q-1)}> K$. But $\rho_0$ is non-increasing function in $(0,T_1)$ which yields 
\begin{equation}\label{esti-con-0}
	\norm{\frac{{h}_0}{\rho_0}}_{L^2((0,T_1)\times \omega)}\leq {K e^{\frac{K_1}{T}}} \norm{u_0}_{H^{-1}(0,1)}.
\end{equation}
Combining  the estimates \eqref{esti-con-0} and \eqref{esti-contr-k},  we obtain 
\begin{align}\label{wght_ctl_est}
\notag	\norm{\frac{h}{\rho_0}}_{L^2((0,T)\times \omega)}
\leq
\sum_{k\geq0}\norm{\frac{{h}_{k}}{\rho_0}}_{L^2((T_{k},T_{k+1})\times \omega )}
&\leq {K e^{\frac{K_1}{T}}} \norm{u_0}_{H^{-1}(0,1)}+ C e^{CT} \sum_{k\geq0} \norm{\frac{f}{\rho_{\F}}}_{L^2(T_k,T_{k+1}; (H^{3}\cap H^2_0)')}\\
	&\leq C  e^{C(T+\frac{1}{T})}  \left(\norm{u_0}_{H^{-1}(0,1)}+\norm{\frac{f}{\rho_{\F}}}_{L^2(0,T;(H^{3}\cap H^2_0)')}\right),
\end{align}
for some constant $C>0$ independent in $T$. 

\noindent 
$\bullet$ {\em Weighted estimate of the solution.}
Let us set 
\begin{equation}
	(u,v)=(\widetilde{u},\widetilde{v})+(\bar{u},\bar{v}),
\end{equation}
which  satisfies the following  control system with the source term $f$, given by 
\begin{equation}\label{System-source fixed}
	\begin{cases}
		u_t+\gamma_1 u_{xxxx}+u_{xxx}+\gamma_2 u_{xx} = f + a v+\chi_{\omega} h_k, &\text{ in }\ (T_k,T_{k+1})\times(0,1),\\
		-v_{xx}+ c v= b u, & \text{ in }\ (T_k,T_{k+1})\times(0,1),\\ 
		u(t, 0)=0,  \  u(t, 1)=0,  & \text{ in }\ (T_k,T_{k+1}),\\
		u_x(t,0)=0, \  u_x(t,1)=0, & \text{ in }\ (T_k,T_{k+1}),\\
		v(t, 0)=0,  \  v(t, 1)=0, & \text{ in }\ (T_k,T_{k+1}),\\
	 u(T_k, x) = m_k, & \text{ in } (0, 1).
	\end{cases}
\end{equation}
for all $k\geq0$. Note that, the solution $(u,v)$ satisfies
\begin{equation*}
	u(T_0)=m_0
\end{equation*}
and, using \eqref{controllability-kth} one has  for all $k\geq0$, that
\begin{align*}
	u(T^-_{k+1})=\widetilde{u}(T^-_{k+1}) + \bar{u}(T^-_{k+1}) = m_{k+1}, \\
	u(T^+_{k+1}) = \widetilde{u}(T^+_{k+1}) + \bar{u}(T^+_{k+1}) = m_{k+1}.
\end{align*}
This gives the continuity of the component $u$ at each $T_{k+1}$ for $k\geq 0$, more precisely 
$$  u \in  \C^0([T_{k},T_{k+1}]; H^{-1}(0,1)). $$
Once we have this, the elliptic part $v$ enjoys
$$v(T_{k}) =  b\left(-\partial_{xx} + c I \right)^{-1} u(T_{k}) \in H^{-1}(0,1), \quad \forall k\geq 0,$$
which yields  the continuity of $v$ at time $T_k$ ($k\geq 0$). 
Moreover, we have the following estimate (thanks to \Cref{prop-esti-regular}) for all $k\geq 0$: 
\begin{align}\label{esti-sol-t-k}
	&\norm{u}_{\C^0([T_{k},T_{k+1}];H^{-1}) \cap L^2(T_{k},T_{k+1};{H^1_0}) \cap H^1(T_k,T_{k+1}; (H^{3}\cap H^2_0)')  }
 + \norm{v}_{\C^0([T_{k},T_{k+1}];H^{-1}) \cap L^2(T_k, T_{k+1}; (H^2\cap H^1_0))}  \notag \\
	&\leq  C e^{CT} \left(\norm{m_{k}}_{H^{-1}(0,1)}+\norm{{h}_{k}}_{L^2((T_{k},T_{k+1})\times \omega)}+\norm{f}_{L^2(T_{k},T_{k+1};(H^{3}\cap H^2_0)')}\right).  
	\end{align}
We start with $k\geq 1$; using the estimates of $m_k$ and $h_k$ from \eqref{esti-m-k} and \eqref{control_estimate tk} (resp.), we get 
	\begin{align*} 
	&	\norm{u}_{\C^0([T_{k},T_{k+1}];H^{-1}) \cap L^2(T_{k},T_{k+1};{H^1_0}) \cap H^1(T_k,T_{k+1}; (H^{3}\cap H^2_0)')  }
		+ \norm{v}_{\C^0([T_{k},T_{k+1}];H^{-1}) \cap L^2(T_k, T_{k+1}; (H^2\cap H^1_0))} \\
	&\leq Ce^{CT} e^{\frac{K}{T_{k+1}-T_{k}}}\norm{f}_{L^2(T_{k-1},T_{k}; (H^{3}\cap H^2_0)')}+ Ce^{CT}  \norm{f}_{L^2(T_{k},T_{k+1}; (H^{3}\cap H^2_0)')}\\
	&\leq Ce^{CT} \rho_\F(T_{k-1}) e^{\frac{K}{T_{k+1}-T_{k}}}\norm{\frac{f}{\rho_\F}}_{L^2(T_{k-1},T_{k+1}; (H^{3}\cap H^2_0)')},
\end{align*}
 since $\rho_{\F}$ is non-increasing in $(0,T)$.   But $\rho_0(T_{k+1})=e^{\frac{K}{T_{k+1}-T_{k}}}\rho_{\F}(T_{k-1})$  for all $k\geq 1$ and 
   therefore
\begin{align}\label{esti-sol-k-k+1}
	& \norm{\frac{u}{\rho_0}}_{\C^0([T_{k},T_{k+1}];H^{-1}) \cap L^2(T_{k},T_{k+1};{H^1_0}) \cap H^1(T_k,T_{k+1}; \left(H^{3}\cap H^2_0\right)')  }
+ \norm{\frac{v}{\rho_0}}_{\C^0([T_{k},T_{k+1}];H^{-1}) \cap L^2(T_k, T_{k+1}; (H^2\cap H^1_0))}  \notag \\
 & \leq   Ce^{CT} \norm{\frac{f}{\rho_{\F}}}_{L^2(T_{k-1},T_{k+1};(H^{3}\cap H^2_0)')}, \qquad \forall k\geq 1, \,\, \text{ since } \rho_0 \text{ is also non-increasing  in } (0,T).
\end{align}

Finally, using the estimate of $h_0$ from \eqref{esti-h_0}, one can deduce that 
\begin{align*}
	& \norm{u}_{\C^0([0,T_{1}];H^{-1}) \cap L^2(0,T_{1};{H^1_0}) \cap H^1(0,T_{1}; (H^{3}\cap H^2_0)')}
	+ \norm{v}_{\C^0([0,T_{1}];H^{-1}) \cap L^2(0, T_{1}; (H^2\cap H^1_0))} \\
	& \leq   C e^{CT}  e^{\frac{K_1}{T}} \rho_0(T_1) \left(\|u_0\|_{H^{-1}(0,1)} + \|f\|_{L^2(0,T_1; (H^{3}\cap H^2_0)')} \right).
\end{align*}
But $\rho_0$ is non-increasing function and it is easy to observe that \begin{align*}\displaystyle \|f\|_{L^2(0,T_1; (H^{3}\cap H^2_0)^\prime)} \leq \norm{\frac{f}{\rho_\F}}_{L^2(0,T_1; (H^{3}\cap H^2_0)^\prime)},
\end{align*}
 which  leads to the following:
\begin{align}\label{esti-sol-0-T_1}
	& \norm{\frac{u}{\rho_0}}_{\C^0([0,T_{1}];H^{-1}) \cap L^2(0,T_{1};{H^1_0}) \cap H^1(0,T_{1}; (H^{3}\cap H^2_0)'))  }
+ \norm{\frac{v}{\rho_0}}_{\C^0([0,T_{1}];H^{-1}) \cap L^2(0, T_{1}; (H^2\cap H^1_0))} \notag  \\
& \leq   Ce^{CT} e^{\frac{K_1}{T}}\left(\|u_0\|_{H^{-1}(0,1)} + \norm{\frac{f}{\rho_\F}}_{L^2(0,T_1; (H^{3}\cap H^2_0)')} \right). 
\end{align}
Combining the estimates \eqref{esti-sol-k-k+1} and \eqref{esti-sol-0-T_1}, we obtain
\begin{align}\label{esti-sol-0-T}
	& \norm{\frac{u}{\rho_0}}_{\C^0([0,T];H^{-1}) \cap L^2(0,T;{H^1_0}) \cap H^1(0,T; (H^{3}\cap H^2_0)')  }
	+ \norm{\frac{v}{\rho_0}}_{\C^0([0,T];H^{-1}) \cap L^2(0, T; (H^2\cap H^1_0))} \notag  \\
	& \leq  C e^{C(T+\frac{1}{T})} \left(\|u_0\|_{H^{-1}(0,1)} + \norm{\frac{f}{\rho_\F}}_{L^2(0, T; (H^{3}\cap H^2_0)')} \right).
\end{align}
In the above and in \eqref{wght_ctl_est}, we denote $\displaystyle C_1:= C e^{C(T+\frac{1}{T})}$ and in what follows, we have the required weighted estimate  \eqref{estimate-weighted}.
The proof is complete.
\end{proof}

\subsubsection{\bf{Application of Banach fixed point theorem}}
Recall that our nonlinear system when a control  acts in KS equation is:
\begin{equation}\label{non lin}
	\begin{cases}
		u_t+\gamma_1 u_{xxxx}+u_{xxx}+\gamma_2 u_{xx}+uu_x= a v+\chi_{\omega} h, &\text{  in } Q_T,\\
		-v_{xx}+ c v= b u, & \text{ in } Q_T,\\ 
	u(t, 0)=0,  \  u(t, 1)=0,  & t \in (0, T),\\
		u_x(t,0)=0, \  u_x(t,1)=0, & t \in (0, T),\\
		v(t, 0)=0,  \  v(t, 1)=0, & t \in (0, T),\\
	u(0, x)=u_0(x), & x \in  (0, 1).
	\end{cases}
\end{equation}
Let us formally test the equation of $(u,v)$ in \eqref{non lin} by any $(\vphi, \theta)\in \big[H^4(0,1)\cap H^2_0(0,1)\big]\times \big[ H^2(0,1)\cap H^1_0(0,1) \big]$, leading to
\begin{multline}
	\frac{d}{dt} \int_0^1 u \vphi + \gamma_1 \int_0^1 u \vphi_{xxxx} - \int_0^1 u \vphi_{xxx} - \frac{1}{2} \int_0^1 u^2 \vphi_x  - \int_0^1 v \theta_{xx} + c\int_0^1 v \theta \\
	 = a\int_0^1 v \vphi + b \int_0^1 u \theta + \int_\omega h \vphi .
\end{multline}
We now define the nonlinearity $uu_x$ in terms of the following function: $F: L^2(0,1) \to H^{-2}(0,1)$ such that
\begin{align}\label{function_nonlinear} 
\left\langle F(u), \vphi \right\rangle_{H^{-2}, H^2_0} = - \frac{1}{2} \int_0^1 u^2 \vphi_x , \quad \forall \vphi \in H^2_0(0,1).
\end{align}
%
%
%
%
Now, recall that the system \eqref{System-source} has unique solution $(u,v)$ given by \eqref{sol-u}--\eqref{sol-v} and moreover, $u\in L^4(0,T; L^2(0,1))$ (see \eqref{embedding}). 
Therefore, one has 
\begin{align*}
	\left|  \left\langle F(u), \vphi \right\rangle_{H^{-2}, H^2_0}  \right| \leq \frac{1}{2}\|\vphi_x\|_{L^\infty(0,1)} \|u(t)\|^2_{L^2(0,1)} \leq C_2 \|\vphi\|_{H^2_0(0,1)} \|u(t)\|^2_{L^2(0,1)}  ,
	\end{align*}
for some constant $C_2>0$, 
since $H^1_0(0,1)\hookrightarrow L^\infty(0,1)$. This yields 
\begin{align}\label{Well-defined_F}
	\|F(u)\|_{L^2(0,T; H^{-2}(0,1))} \leq C_2 \|u\|^2_{L^4(0,T; L^2(0,1))} ,
\end{align}
that is the function $F$ is well-defined. In what follows, we  define the map 
\begin{align}\label{map-fixed_point}
	\Lambda &: \F \to \F \\
	        &  f \mapsto -F(u) \notag, 
\end{align} 
where $u$ is the solution of the system \eqref{System-source}.
 We also consider an open ball $B(0,R)$ in $\F$ with center $0$ and radius $R>0$. We begin with the following lemma. 
\begin{lemma}[Stability]\label{Lemma-stability} 
 There exists some $R>0$ such that $\overline{B(0,R)} \subset \F$ is stable under the map $\Lambda$. In other words, $\Lambda\big(\overline{B(0,R)}\big) \subset \overline{B(0,R)}$.
\end{lemma}
\begin{proof} 
Using the definition \eqref{function_nonlinear}, we have  $\forall \vphi\in H^2_0(0,1)$, 
\begin{align*}
	\left|\left\langle \frac{F(u)}{\rho_\F}, \vphi \right\rangle_{H^{-2}, H^2_0}\right| &= \frac{1}{2} \left|\int_0^1 \frac{u^2}{\rho_F} \vphi_x\right| \\
	   & \leq \frac{1}{2} \left\|\frac{\rho_0^2}{\rho_\F} \right\|_{L^\infty(0,T)} \left\|\frac{u}{\rho_0}  \right\|_{H^1_0(0,1)} \left|\int_0^1 \frac{u}{\rho_0} \vphi_x \right|  .
   \end{align*} 
But $\displaystyle \rho^2_0(t)/ \rho_\F(t)\leq 1$ for all $t\in [0,T]$ and thus there exists some constant $C_2>0$ such that
\begin{align}\label{norm_F-weighted}
	\left\|\frac{F(u)}{\rho_\F}\right\|_{H^{-2}(0,1)} \leq C_2 \left\|\frac{u}{\rho_0}  \right\|_{H^1_0(0,1)} \left\|\frac{u}{\rho_0}  \right\|_{H^{-1}(0,1)}.
\end{align}
Hence, for any source term $f\in \overline{B(0,R)}$, one has 
\begin{align} 
\left\|\frac{\Lambda(f)}{\rho_\F} \right\|_{L^2(0,T; H^{-2}(0,1))} 
 =  &	\left\|\frac{F(u)}{\rho_\F} \right\|_{L^2(0,T; H^{-2}(0,1))}   \notag \\
\leq  & C_2 \left\|\frac{u}{\rho_0}  \right\|_{\C^0([0,T]; H^{-1}(0,1))} \left\|\frac{u}{\rho_0}  \right\|_{L^2(0,T; H^1_0(0,1))}  \label{eq-1} \\  
\label{eq-2}
\leq &  C_1^2C_2 \left(\|u_0\|_{H^{-1}(0,1)} + \left\|\frac{f}{\rho_\F}\right\|_{L^2(0,T; (H^{3}\cap H^2_0)^\prime)}      \right)^2 \\  \label{eq-3}
 \leq &  4C^2_1C_2 R^2,
\end{align} 
   where we have used the estimate  \eqref{estimate-weighted} to obtain the equation \eqref{eq-2} from \eqref{eq-1}, and then in \eqref{eq-2} we consider the initial data $u_0$ such that $\|u_0\|_{H^{-1}(0,1)}\leq R$.
 Now, consider 
   \begin{align}\label{definition-R}
   	 R = \frac{1}{8C^2_1C_2},
   \end{align} 
so that, from  \eqref{eq-3}, we have 
\begin{align*}
\left\|\frac{\Lambda(f)}{\rho_\F} \right\|_{L^2(0,T; H^{-2}(0,1))} \leq R/ 4\leq R ,
\end{align*}
yielding that $\overline{B(0,R)}$ is invariant under the map $\Lambda$, defined in \eqref{map-fixed_point}. 
\end{proof} 

We also have the following result. 
\begin{lemma}[Contraction]\label{Lemma-contraction}
	The map $\Lambda$ defined by  \eqref{map-fixed_point} is a contraction map on the closed ball $\overline{B(0,R)}$. 
\end{lemma}
\begin{proof} 
Consider $f_1, f_2 \in \overline{B(0,R)}$. Then in view of \Cref{Proposition-weighted}, there exists control $h_j \in \V$ and  solution $(u_j,v_j)$ to the system \eqref{System-source} associated with the source term $f_j$ for  $j=1,2$. 
%
     
By the linearity of the solution map (thanks to \Cref{Proposition-weighted}),  we have by means of the estimate \eqref{estimate-weighted}), 
\begin{multline}\label{diff-estimate}
	\left\|\frac{u_1-u_2}{\rho_0} \right\|_{\C^0(H^{-1}) \cap L^2(H^1_0) \cap H^1((H^{3}\cap H^2_0)')} + 	\left\|\frac{v_1-v_2}{\rho_0} \right\|_{L^2(H^2\cap H^1_0)} + \left\|\frac{h_1-h_2}{\rho_0} \right\|_{L^2((0,T)\times \omega)} \\
	\leq C_1  \left\|\frac{f_1-f_2}{\rho_\F} \right\|_{L^2((H^{3}\cap H^2_0)')} .
\end{multline}
Now, using \eqref{function_nonlinear}, we have $\forall \vphi\in H^2_0(0,1)$, 
\begin{align*}
	\left|\left\langle \frac{F(u_1)-F(u_2)}{\rho_\F}, \vphi\right\rangle_{H^{-2}, H^2_0}\right| &\leq  \frac{1}{2}\left|\frac{\rho_0^2(t)}{\rho_\F(t)} \right| \left| \int_0^1 \left(\frac{u^2_1 - u^2_2}{\rho_0^2} \right)\vphi_x\right| \\
	&\leq  \frac{1}{2}\left(\left\| \frac{u_1}{\rho_0}\right\|_{H^1_0(0,1)} + \left\|\frac{u_2}{\rho_0}\right\|_{H^1_0(0,1)}\right) \left|\int_0^1 \frac{(u_1-u_2)}{\rho_0} \vphi_x \right| ,
\end{align*}
and so,
\begin{align*}
	\left\|\frac{F(u_1)-F(u_2)}{\rho_\F}\right\|_{H^{-2}(0,1)} \leq C_2 \left(\left\| \frac{u_1}{\rho_0} \right\|_{H^1_0(0,1)} + \left\|\frac{u_2}{\rho_0}\right\|_{H^1_0(0,1)}\right) \left\|\frac{u_1-u_2}{\rho_0}\right\|_{H^{-1}(0,1)},
\end{align*}
where the constant  $C_2$ is the one which has  appeared in \eqref{norm_F-weighted}. 
Using this, we have 
\begin{align}
	&\left\| \frac{\Lambda(f_1) - \Lambda(f_2)}{\rho_\F} \right\|_{L^2(0,T; H^{-2}(0,1))}    \notag \\
	 = &  	\left\| \frac{F(u_1) - F(u_2)}{\rho_\F} \right\|_{L^2(0,T;H^{-2}(0,1))}  \notag \\
	 \label{tag-1}
	 \leq  &  C_2 \left(\left\| \frac{u_1}{\rho_0} \right\|_{L^2(0,T;H^1_0(0,1))} + \left\|\frac{u_2}{\rho_0}\right\|_{L^2(0,T;H^1_0(0,1))}\right) \left\|\frac{u_1-u_2}{\rho_0}\right\|_{\C^0([0,T]; H^{-1}(0,1))}\\
	 \label{tag-2}
	 \leq & C_2 \left(2C_1\|u_0\|_{H^{-1}(0,1)} + C_1\left\|\frac{f_1}{\rho_\F} \right\|_{L^2((H^{3}\cap H^2_0)')} + C_1\left\|\frac{f_2}{\rho_\F} \right\|_{L^2((H^{3}\cap H^2_0)')}   \right) C_1 \left\|\frac{f_1-f_2}{\rho_\F} \right\|_{L^2((H^{3}\cap H^2_0)')}  \\
	 \label{tag-3}
	  \leq & 4 C_1^2 C_2  R \left\|\frac{f_1-f_2}{\rho_\F} \right\|_{L^2((H^{3}\cap H^2_0)')} \\
	 \leq & \frac{1}{2} \left\|\frac{f_1-f_2}{\rho_\F} \right\|_{L^2((H^{3}\cap H^2_0)')}  \label{tag-4},
\end{align}
where we make use of the estimate \eqref{estimate-weighted} to deduce the inequality \eqref{tag-2} from \eqref{tag-1} and then in \eqref{tag-2}, we used  $f_1, f_2 \in \overline{B(0,R)}$ and $\|u_0\|_{H^{-1}(0,1)}\leq R$. Finally, thanks to the choice of $R$ in \eqref{definition-R}, we get \eqref{tag-4} leading to the fact that the map $\Lambda$ is a contraction map on $\overline{B(0,R)}$. 
\end{proof}
\subsubsection{\bf
	Proof of \Cref{Thm-nonlinear-KS} }
\begin{proof}
The proof is followed by the \Cref{Lemma-stability} and \ref{Lemma-contraction}. Thanks to those results, 
we can apply
 the Banach fixed point argument which ensures the existence of a unique fixed point of the map $\Lambda$ in $\overline{B(0,R)}\subset \F$.  Denote the fixed point by $f_0 \in \F$.   
 
 Now, in terms of  to \Cref{Proposition-weighted}, with the above $f_0 \in \overline{B(0,R)}$ and initial data $\|u_0\|_{H^{-1}(0,1)}\leq R$, there exists a control $h\in \V$ that satisfies the estimate \eqref{estimate-weighted}.   Then, the property $\displaystyle \lim_{t\to T^{-}} \rho_0(t)=0$ forces that 
 $$ \left(u(T,x), v(T,x)  \right) =(0,0) , \quad \forall \,  x \, \in (0,1).$$
 This proves the local null-controllability of the nonlinear system \eqref{nonlinear-KS}--\eqref{bd}--\eqref{in}.  
\end{proof}

\section{Control acting in the elliptic equation}\label{Section-elliptic}

In this section, we study the controllability property of the system \eqref{nonlinear-elliptic}--\eqref{bd}--\eqref{in}, that is when a control locally acts  only in the elliptic equation.   
As earlier, a suitable Carleman estimate is required to get an observability inequality for the linearized model \eqref{main-2}--\eqref{bd}--\eqref{in}.

\subsection{A global Carleman estimate}\label{sec-carleman-2}

 We have the following estimate.

\begin{theorem}[Carleman estimate: control in elliptic eq.]\label{Thm_Carleman_main 2}
	Let the weight functions $(\vphi, \xi)$  be given by \eqref{weight_function}. Then, there exist positive constants $\lambda_0$, $s_0:=  \mu_0(T+T^{2})$ for some $\mu_0$ and $C$ which  depend on $\gamma_1, \gamma_2, a,b,c$ and the set $\omega$,  such that we have the following estimate satisfied by the solution to \eqref{Adjoint-system-2},
	\begin{align}\label{carle-joint 2} 
		I_{KS}(s,\lambda; \sigma) +   	I_{E}(s,\lambda; \psi)
		\leq Cs^{11}\lambda^{12}\int_0^T\int_{\omega} e^{-2s\vphi} \xi^{11} |\psi|^2,
	\end{align}
	for all $\lambda \geq \lambda_0$ and $s\geq s_0$, where $I_{KS}(s,\lambda; \sigma)$ and $I_E(s,\lambda; \psi)$ are introduced by \eqref{Notation_KS} and \eqref{Notation_elliptic} respectively. 
\end{theorem}
\begin{proof} 
 Adding the two individual Carleman estimates for KS and elliptic parts, we have the following estimate (as obtained in \eqref{Add_carlemans-2}), 
\begin{align}\label{Add_carlemans-3-2}
	I_{KS}(s,\lambda; \sigma) + I_{E}(s,\lambda; \psi) 
	\leq C \bigg[
	s^7\lambda^8 \int_0^T\int_{\omega_0} e^{-2s\vphi} \xi^7 |\sigma|^2 
	+ s^3\lambda^4 \int_0^T\int_{\omega_0} e^{-2s\vphi} \xi^3 |\psi|^2\bigg],
\end{align}
for all $\lambda\geq \lambda_0$ and $s\geq \mu_0(T+T^2)$ where $\lambda_0$ and $\mu_0$ are as introduced in the Step 1 of the proof of \Cref{Thm_Carleman_main}. 

Thus,  we just need to absorb the integral related to $\sigma$ from the right hand side of \eqref{Add_carlemans-3-2}.

\smallskip 

\noindent 
	{\bf Absorbing the observation integral in $\psi$.}
Recall	 the adjoint system \eqref{Adjoint-system-2}, one has (since $a\neq 0$),
	\begin{align*}
		\sigma = \frac{1}{a} \left(  -\psi_{xx}+c\psi \right), \quad \text{in } Q_T.
	\end{align*}
	We also
recall the smooth function $\phi \in \C^\infty_c(\omega_1)$ as given by \eqref{smooth_func}.
	
Using these, we have 
\begin{align}\label{Obs-psi 1}
	s^7\lambda^8 \int_0^T\int_{\omega_0} e^{-2s\vphi} \xi^7 |\sigma|^2 
	&\leq s^7\lambda^8 \int_0^T\int_{\omega_1} \phi e^{-2s\vphi} \xi^7 |\sigma|^2 \\ \notag 
	&= \frac{1}{a}s^7\lambda^8 \int_0^T\int_{\omega_1} \phi e^{-2s\vphi} \xi^7 \sigma \left( -\psi_{xx}+c\psi   \right) \\  \notag 
	&=B_1+ B_2.
\end{align}
(i) Let us  first calculate $B_1$. 
\begin{align*}
	B_1=&\frac{1}{a}s^7\lambda^8\int_0^T\int_{\omega_1}\bigg[ \phi_x e^{-2s\vphi} \xi^7 \sigma \psi_x+\phi (e^{-2s\vphi} \xi^7)_x \sigma \psi_x+\phi_x e^{-2s\vphi} \xi^7 \sigma_x \psi_x\bigg]\\
	=&-\frac{1}{a}s^7\lambda^8\int_0^T\int_{\omega_1}\bigg[ \left(\phi_{xx} e^{-2s\vphi} \xi^7 + 2\phi_x (e^{-2s\vphi} \xi^7)_x +\phi (e^{-2s\vphi} \xi^7)_{xx}\right) \sigma \psi 
	\\ 
	& \quad + \phi e^{-2s\vphi} \xi^7 \sigma_{xx} \psi
	+\left(2\phi_x  e^{-2s\vphi} \xi^7 + 2\phi (e^{-2s\vphi} \xi^7)_x \right)\sigma_x \psi  \bigg].
\end{align*}
Next, using the bounds \eqref{weight-deri-x} and  applying the Cauchy-Schwarz and Young's inequalities successively, we get 
\begin{align*}
|B_1|&\leq C s^7\lambda^8 \int_0^T\int_{\omega_1} (e^{-2s\vphi} \xi^7) |\sigma| |\psi|+C s^8\lambda^9 \int_0^T\int_{\omega_1} (e^{-2s\vphi} \xi^8) |\sigma| |\psi|\\
&\quad +C s^9\lambda^{10} \int_0^T\int_{\omega_1} (e^{-2s\vphi} \xi^9) |\sigma| |\psi|+C s^7\lambda^{8} \int_0^T\int_{\omega_1} (e^{-2s\vphi} \xi^7) |\sigma_{xx}| |\psi|\\
&\quad\quad + C s^7\lambda^{8} \int_0^T\int_{\omega_1} (e^{-2s\vphi} \xi^7) |\sigma_x| |\psi|+ C s^8\lambda^9 \int_0^T\int_{\omega_1} (e^{-2s\vphi} \xi^8) |\sigma_x| |\psi|
\\
&\leq 3\epsilon  s^7\lambda^8 \iintQ e^{-2s\vphi} \xi^7 |\sigma|^2  + \epsilon s^3 \lambda^4 \iintQ e^{-2s\vphi} \xi^3 |\sigma_{xx}|^2 + 2\epsilon s^5\lambda^6 \iintQ e^{-2s\vphi} \xi^5 |\sigma_{x}|^2 \\
& \quad + \frac{C}{\epsilon} \int_0^T\int_{\omega_1} e^{-2s\vphi} \left(s^7\lambda^8 \xi^7 + s^9\lambda^{10} \xi^9 + s^{11} \lambda^{12} \xi^{11}\right) |\psi|^2 \\
&\leq 3\epsilon I_{KS}(s,\lambda; \sigma) +\frac{C}{\epsilon} s^{11}\lambda^{12} \int_0^T\int_{\omega_1} e^{-2s\vphi} \xi^{11} |\psi|^2.
\end{align*}

\noindent 
(ii) Estimate of $B_2$ can be taken care in a similar manner.  We then fix a small $\epsilon>0$  such that there is a positive constant $C$ with the following:
\begin{align*}
	I_{KS}(s,\lambda; \sigma) + I_{E}(s,\lambda; \psi) 
	\leq C s^{11} 
	\lambda^{12 }\int_0^T\int_{\omega} (e^{-2s\vphi} \xi^{11}) |\psi|^2,
\end{align*}
for all $\lambda\geq \lambda_0$ and $s\geq \mu_0(T+T^2)$. 

This ends the proof. 
\end{proof}

\subsection{Observability inequality and null-controllability of the linearized model}
Using the Carleman estimate \eqref{carle-joint 2}, one can prove the following observability inequality. 
\begin{proposition}[Observability inequality: control in elliptic eq.]\label{prop:refined}
	There exists a positive   constant $C$  depending on $\omega$, $a$, $b$, $c$, $\gamma_1$, $\gamma_2$ and $N$ but not $T$,  such that
	we have the following observability inequality
	\begin{align}\label{Partial-obser-2}
		\|\sigma(0,\cdot)\|^2_{H^2_0(0,1)} + 	\|\psi(0,\cdot)\|^2_{H^2(0,1)}\leq
		{ C e^{C/T}}\int_0^T\int_{\omega}  |\psi|^2  ,
	\end{align}
	where $(\sigma,\psi)$ is the solution to the adjoint system \eqref{Adjoint-system-2}.
	\end{proposition}
\begin{proof}
The proof of above proposition will be done  in the same spirit of the proof of  \Cref{prop:refined_a_0}.  We skip the details. 
\end{proof}

The proof of required controllability result, i.e., \Cref{Thm-linear-control-KS}--\Cref{point-2} will be followed by a  similar argument as we addressed  in Section \ref{Section-obser}. 
%
%

\subsection{Null-controllability of the nonlinear system: proof of \Cref{Thm-nonlinear-Ellp}}  To prove the null-controllability of the system \eqref{nonlinear-elliptic}--\eqref{bd}--\eqref{in}, we will follow the same procedure as described in Section \ref{Section-nonlinear-KS}. We shall omit the details for this case since the computations are similar as of Section \ref{Section-nonlinear-KS}. 

In what follows,  recall the weight functions $\rho_0$, $\rho_\F$ and the spaces $\F$, $\Y$ from Section \ref{Section-nonlinear-KS}.  Then, we consider the following system with a source term $f\in \F$,
\begin{equation}\label{System-source-2}
	\begin{cases}
		u_t+\gamma_1 u_{xxxx}+u_{xxx}+\gamma_2 u_{xx} = f+ a v &\text{in } Q_T,\\
		-v_{xx}+ c v= b u +\chi_{\omega} h , & \text{in } Q_T,\\ 
		u(t, 0)=0,  \  u(t, 1)=0,  & t \in (0, T),\\
		u_x(t,0)=0, \  u_x(t,1)=0 & t \in (0, T)\\
		v(t, 0)=0,  \  v(t, 1)=0 & t \in (0, T),\\
		u(0, x)=u_0(x), & x \in  (0, 1),
	\end{cases}
\end{equation}	
with $u_0 \in H^{-1}(0,1)$. 

Similar to \Cref{Proposition-weighted}, one can prove the following results. 
\begin{proposition}\label{Proposition-weighted-2}
	Let $T>0$,  $u_0 \in H^{-1}(0,1)$  and $f\in \F$ be given.  We also take any $(a,b)\in \mathbb R^2$ with $a\neq 0$. Then, there exists a linear map 
	\begin{align}\label{map-linearity-2}
		\left[u_0, f\right] \in H^{-1}(0,1) \times \F \mapsto \left[(u,v), h  \right] \in \Y \times \V,
	\end{align}
	such that $\left[(u,v), h  \right]$ solves \eqref{System-source-2}. 
	
	Moreover, we have the following regularity estimate
	\begin{multline}\label{estimate-weighted-2}
		\left\|\frac{u}{\rho_0} \right\|_{\C^0(H^{-1}) \cap L^2(H^1_0) \cap H^1((H^{3}\cap H^2_0)')} + 	\left\|\frac{v}{\rho_0} \right\|_{L^2(H^2\cap H^1_0)} + \left\|\frac{h}{\rho_0} \right\|_{L^2((0,T)\times \omega)}\\
		\leq  \widehat C_1  \left(\|u_0\|_{H^{-1}(0,1)} + \left\|\frac{f}{\rho_\F} \right\|_{L^2((H^{3}\cap H^2_0)')} \right),
	\end{multline}
	where the constant $\widehat C_1>0$ does not depend on $u_0$, $f$ or $h$. 
\end{proposition}

Thereafter, the Banach fixed point argument will give us the required local null-controllability result  of the system \eqref{nonlinear-elliptic}--\eqref{bd}--\eqref{in}, that is  \Cref{Thm-nonlinear-Ellp}.   We skip the details here since it can be done in a similar fashion as in the previous case.

\section{Conclusions and open questions}\label{Sec-Conclusion}
Let us  briefly recall  the whole study of the present article. We have considered a coupled parabolic-elliptic system containing a fourth-order parabolic equation, namely KS-KdV and a second-order elliptic equation. Local distributed null-controllability of the aforementioned model has been proved by means of a localized interior control acting only in one equation where the nonlinearity is $uu_x$.

In both cases, we have established  global Carleman estimate which gives rise to some suitable observability inequality and then the usual duality argument provides the null-controllability for the associated linearized models. To this end, we employed the source term method and Banach fixed point argument to obtain the local-null controllability for the nonlinear systems.

\smallskip 

One can think of some other issues which has been studied throughout years for various model. Let us mention some of  those which can lead some future research direction.

\smallskip 

\noindent
$\bullet$ {\em Boundary null-controllability.} In \cite{hernandez2021boundary}, the authors studied the boundary local null-controllability of some nonlinear parabolic-elliptic system by means of a single control force acting at the left end of the Dirichlet boundary of parabolic component. They handled the linear model by using a boundary Carleman estimate and then the  local inversion technique has been applied to conclude the nonlinear system.

Coming to the KS equation, it is known that to deal with the boundary controllability  with single control force, one needs to put some restriction on the anti-diffusion parameter $\gamma_2$.  This issues can be overcome by introducing another control in the dynamics (see \cite{CE10}, \cite{Cer17} for more details). Moreover, in \cite{CE12}, the authors studied the local null-controllability of stabilized KS by means of three boundary controls. In the spirit of these aforementioned works, it will be interesting to consider the system with less number of control: 
\begin{equation}
	\begin{cases}
		u_t+\gamma_1 u_{xxxx}+u_{xxx}+\gamma_2 u_{xx} + uu_x = a v, &\text{ in } Q_T,\\
		-v_{xx}+ c v= b u, & \text{ in } Q_T, \\
		u(0,x)=u_0(x), & \text{ in } (0,1),
	\end{cases} 
\end{equation}
with any of the following boundary control data
\begin{equation}
	\begin{cases}
		u(t, 0)=q_0(t),  \  u(t, 1)=0,  & t \in (0, T), \\ 
		u_x(t,0)=0, \  u_x(t,1)=0, & t \in (0, T),    \\
		v(t, 0)=0,  \  v(t, 1)=0, & t \in (0, T), 
	\end{cases}
	\quad \text{ or }  \ 
	\begin{cases}u(t, 0)=0,  \  u(t, 1)=0,  & t \in (0, T), \\ 
		u_x(t,0)=0, \  u_x(t,1)=0, & t \in (0, T),    \\
		v(t, 0)=p_0(t),  \  v(t, 1)=0, & t \in (0, T).
	\end{cases}
\end{equation}

\smallskip 

\noindent
$\bullet$ {\em Null-controllability in 2D setting.} 
Controllability of KS-KdV-elliptic equation in higher dimension (in particular 2D) would be an interesting problem to study.  For the case of Chemotaxis model (second-order nonlinear parabolic-elliptic equation) \cite{GB14}, the authors obtained  the local null-controllability result in  bounded domain $\Omega \subset \mathbb{R}^N, N\geq 1$.  For the KS equation, the author in  \cite{Taka17} studied the local null-controllability in 2D by means of a boundary control by assuming some conditions on the boundary. Thanks to the recent result \cite{JR20} concerning the spectral inequality for the bi-Laplace operator, the null-controllability of the following system
\begin{equation}\label{KS-2d}
	\begin{cases}
		u_t+\Delta^2 u=\chi_{\omega}h,& (t,x) \in (0,T)\times \Omega,\\
		u=0, \ \  	\dfrac{\partial u}{\partial \nu}=0, &(t,x) \in (0,T)\times \partial\Omega,\\
		u(0,x)=u_0(x), & x\in \Omega,
	\end{cases}
\end{equation} 
can be established.  On the light of this discussion, it is reasonable to pose a local controllability problem regarding KS-elliptic model in 2D, for instance:
\begin{equation}
	\begin{cases}
		u_t+\Delta^2 u+ \nu \Delta u+ \frac{1}{2}\abs{\nabla u}^2= v+\chi_{\omega}g,& (t,x) \in (0,T)\times \Omega,\\
		-\Delta v+\gamma v=\delta u, & (t,x) \in (0,T)\times \Omega,
	\end{cases}
\end{equation}
with the boundary and initial conditions as in \eqref{KS-2d}.

\section*{Acknowledgements}
{The authors would like to express sincere thanks to the reviewer
for helpful comments, valuable questions and suggestions to improve the first version of the paper.} The work of the first author is partially supported by the French government research program
``Investissements d'Avenir" through the IDEX-ISITE initiative 16-IDEX-0001 (CAP 20-25). 
The second author acknowledges the supports from Department of Atomic Energy and NBHM Fellowship, Grant No. 0203/16(21)/2018-R\&D-II/10708.

\bibliographystyle{plain} 
\bibliography{biblio}

\begin{thebibliography}{10}

\bibitem{KBBT11b}
F.~Ammar-Khodja, A.~Benabdallah, M.~Gonz\'{a}lez-Burgos, and L.~de~Teresa.
\newblock The {K}alman condition for the boundary controllability of coupled
  parabolic systems. {B}ounds on biorthogonal families to complex matrix
  exponentials.
\newblock {\em J. Math. Pures Appl. (9)}, 96(6):555--590, 2011.

\bibitem{ABBTc}
Farid Ammar~Khodja, Assia Benabdallah, Manuel Gonz\'{a}lez-Burgos, and Luz
  de~Teresa.
\newblock Minimal time for the null controllability of parabolic systems: the
  effect of the condensation index of complex sequences.
\newblock {\em J. Funct. Anal.}, 267(7):2077--2151, 2014.

\bibitem{ABBT16}
Farid Ammar~Khodja, Assia Benabdallah, Manuel Gonz\'{a}lez-Burgos, and Luz
  de~Teresa.
\newblock New phenomena for the null controllability of parabolic systems:
  minimal time and geometrical dependence.
\newblock {\em J. Math. Anal. Appl.}, 444(2):1071--1113, 2016.

\bibitem{Barbu}
V.~Barbu.
\newblock Exact controllability of the superlinear heat equation.
\newblock {\em Appl. Math. Optim.}, 42(1):73--89, 2000.

\bibitem{BBB14}
Assia Benabdallah, Franck Boyer, Manuel Gonz\'{a}lez-Burgos, and Guillaume
  Olive.
\newblock Sharp estimates of the one-dimensional boundary control cost for
  parabolic systems and application to the {$N$}-dimensional boundary null
  controllability in cylindrical domains.
\newblock {\em SIAM J. Control Optim.}, 52(5):2970--3001, 2014.

\bibitem{Bhandari21a}
Kuntal Bhandari and Franck Boyer.
\newblock Boundary null-controllability of coupled parabolic systems with
  {R}obin conditions.
\newblock {\em Evol. Equ. Control Theory}, 10(1):61--102, 2021.

\bibitem{KV}
Kuntal Bhandari and Víctor Hernández-Santamaría.
\newblock An insensitizing control problem for a linear stabilized
  kuramoto-sivashinsky system, 2022.

\bibitem{CE16}
Nicol\'{a}s Carre\~{n}o and Eduardo Cerpa.
\newblock Local controllability of the stabilized {K}uramoto-{S}ivashinsky
  system by a single control acting on the heat equation.
\newblock {\em J. Math. Pures Appl. (9)}, 106(4):670--694, 2016.

\bibitem{CE10}
Eduardo Cerpa.
\newblock Null controllability and stabilization of the linear
  {K}uramoto-{S}ivashinsky equation.
\newblock {\em Commun. Pure Appl. Anal.}, 9(1):91--102, 2010.

\bibitem{Cer17}
Eduardo Cerpa, Patricio Guzm\'{a}n, and Alberto Mercado.
\newblock On the control of the linear {K}uramoto-{S}ivashinsky equation.
\newblock {\em ESAIM Control Optim. Calc. Var.}, 23(1):165--194, 2017.

\bibitem{CE11}
Eduardo Cerpa and Alberto Mercado.
\newblock Local exact controllability to the trajectories of the 1-{D}
  {K}uramoto-{S}ivashinsky equation.
\newblock {\em J. Differential Equations}, 250(4):2024--2044, 2011.

\bibitem{CE12}
Eduardo Cerpa, Alberto Mercado, and Ademir~F. Pazoto.
\newblock On the boundary control of a parabolic system coupling {KS}-{K}d{V}
  and heat equations.
\newblock {\em Sci. Ser. A Math. Sci. (N.S.)}, 22:55--74, 2012.

\bibitem{Cerpa-Mercado-Pazoto}
Eduardo Cerpa, Alberto Mercado, and Ademir~F. Pazoto.
\newblock Null controllability of the stabilized {K}uramoto-{S}ivashinsky
  system with one distributed control.
\newblock {\em SIAM J. Control Optim.}, 53(3):1543--1568, 2015.

\bibitem{JPPuel14}
Felipe~Wallison Chaves-Silva, Sergio Guerrero, and Jean~Pierre Puel.
\newblock Controllability of fast diffusion coupled parabolic systems.
\newblock {\em Math. Control Relat. Fields}, 4(4):465--479, 2014.

\bibitem{DLv5}
Robert Dautray and Jacques-Louis Lions.
\newblock {\em Mathematical analysis and numerical methods for science and
  technology. {V}ol. 5}.
\newblock Springer-Verlag, Berlin, 1992.
\newblock Evolution problems. I, With the collaboration of Michel Artola,
  Michel Cessenat and H\'{e}l\`ene Lanchon, Translated from the French by Alan
  Craig.

\bibitem{Duan}
Jinqiao Duan and Vincent~J. Ervin.
\newblock Dynamics of a nonlocal {K}uramoto-{S}ivashinsky equation.
\newblock {\em J. Differential Equations}, 143(2):243--266, 1998.

\bibitem{FR71}
H.~O. Fattorini and D.~L. Russell.
\newblock Exact controllability theorems for linear parabolic equations in one
  space dimension.
\newblock {\em Arch. Rational Mech. Anal.}, 43:272--292, 1971.

\bibitem{FR75}
H.~O. Fattorini and D.~L. Russell.
\newblock Uniform bounds on biorthogonal functions for real exponentials with
  an application to the control theory of parabolic equations.
\newblock {\em Quart. Appl. Math.}, 32:45--69, 1974/75.

\bibitem{FC16}
E.~Fern\'{a}ndez-Cara, J.~Limaco, and S.~B. de~Menezes.
\newblock Controlling linear and semilinear systems formed by one elliptic and
  two parabolic {PDE}s with one scalar control.
\newblock {\em ESAIM Control Optim. Calc. Var.}, 22(4):1017--1039, 2016.

\bibitem{FC13}
E.~Fern\'{a}ndez-Cara, J.~Limaco, and Silvano~B. de~Menezes.
\newblock Null controllability for a parabolic-elliptic coupled system.
\newblock {\em Bull. Braz. Math. Soc. (N.S.)}, 44(2):285--308, 2013.

\bibitem{Cara-Burgos-Teresa}
Enrique Fern\'andez-Cara, Manuel Gonz\'alez-Burgos, and Luz de~Teresa.
\newblock Boundary controllability of parabolic coupled equations.
\newblock {\em J. Funct. Anal.}, 259(7):1720--1758, 2010.

\bibitem{Cara-Guerrero}
Enrique Fern\'{a}ndez-Cara and Sergio Guerrero.
\newblock Global {C}arleman inequalities for parabolic systems and applications
  to controllability.
\newblock {\em SIAM J. Control Optim.}, 45(4):1399--1446, 2006.

\bibitem{CG07}
Enrique Fern\'{a}ndez-Cara and Sergio Guerrero.
\newblock Null controllability of the {B}urgers system with distributed
  controls.
\newblock {\em Systems Control Lett.}, 56(5):366--372, 2007.

\bibitem{FC05}
Enrique Fern\'{a}ndez-Cara and Enrique Zuazua.
\newblock Null and approximate controllability for weakly blowing up semilinear
  heat equations.
\newblock {\em Ann. Inst. H. Poincar\'{e} C Anal. Non Lin\'{e}aire},
  17(5):583--616, 2000.

\bibitem{Fur-Ima}
A.~V. Fursikov and O.~Yu. Imanuvilov.
\newblock {\em Controllability of evolution equations}, volume~34 of {\em
  Lecture Notes Series}.
\newblock Seoul National University, Research Institute of Mathematics, Global
  Analysis Research Center, Seoul, 1996.

\bibitem{OG07}
O.~Glass and S.~Guerrero.
\newblock On the uniform controllability of the {B}urgers equation.
\newblock {\em SIAM J. Control Optim.}, 46(4):1211--1238, 2007.

\bibitem{SGI07}
S.~Guerrero and O.~Yu. Imanuvilov.
\newblock Remarks on global controllability for the {B}urgers equation with two
  control forces.
\newblock {\em Ann. Inst. H. Poincar\'{e} C Anal. Non Lin\'{e}aire},
  24(6):897--906, 2007.

\bibitem{G07}
Sergio Guerrero.
\newblock Null controllability of some systems of two parabolic equations with
  one control force.
\newblock {\em SIAM J. Control Optim.}, 46(2):379--394, 2007.

\bibitem{GB14}
Bao-Zhu Guo and Liang Zhang.
\newblock Local null controllability for a chemotaxis system of
  parabolic-elliptic type.
\newblock {\em Systems Control Lett.}, 65:106--111, 2014.

\bibitem{hernandez2021boundary}
Esteban Hern{\'a}ndez, Christophe Prieur, and Eduardo Cerpa.
\newblock Boundary null controllability of some parabolic-elliptic systems,
  2021.

\bibitem{hernandezsantamaria:hal-03090716}
V{\'i}ctor Hern{\'a}ndez-Santamar{\'i}a, Alberto Mercado, and Piero Visconti.
\newblock {Boundary controllability of a simplified stabilized
  Kuramoto-Sivashinsky system}.
\newblock working paper or preprint, December 2020.

\bibitem{Ioakim}
Xenakis Ioakim.
\newblock Analyticity for a class of nonlocal {K}uramoto-{S}ivashinsky
  equations arising in interfacial electrohydrodynamics.
\newblock {\em Math. Methods Appl. Sci.}, 41(10):3547--3557, 2018.

\bibitem{kumar2022null}
Manish Kumar and Subrata Majumdar.
\newblock Null controllability of the linear stabilized kuramoto-sivashinsky
  system using moment method.
\newblock {\em arXiv preprint arXiv:2205.03638}, 2022.

\bibitem{majumdar:hal-03695906}
Manish Kumar and Subrata Majumdar.
\newblock {On the controllability of a system coupling
  Kuramoto-Sivashinsky-Kortweg-de Vries and transport equations}.
\newblock working paper or preprint, June 2022.

\bibitem{Kevin}
K.~Le~Balc'h.
\newblock Global null-controllability and nonnegative-controllability of
  slightly superlinear heat equations.
\newblock {\em J. Math. Pures Appl. (9)}, 135:103--139, 2020.

\bibitem{JR12}
J\'{e}r\^{o}me Le~Rousseau and Gilles Lebeau.
\newblock On {C}arleman estimates for elliptic and parabolic operators.
  {A}pplications to unique continuation and control of parabolic equations.
\newblock {\em ESAIM Control Optim. Calc. Var.}, 18(3):712--747, 2012.

\bibitem{JR20}
J\'{e}r\^{o}me Le~Rousseau and Luc Robbiano.
\newblock Spectral inequality and resolvent estimate for the bi-{L}aplace
  operator.
\newblock {\em J. Eur. Math. Soc. (JEMS)}, 22(4):1003--1094, 2020.

\bibitem{LR95}
G.~Lebeau and L.~Robbiano.
\newblock Contr\^{o}le exact de l'\'{e}quation de la chaleur.
\newblock {\em Comm. Partial Differential Equations}, 20(1-2):335--356, 1995.

\bibitem{Tucsnak-nonlinear}
Y.~Liu, T.~Takahashi, and M.~Tucsnak.
\newblock Single input controllability of a simplified fluid-structure
  interaction model.
\newblock {\em ESAIM Control Optim. Calc. Var.}, 19(1):20--42, 2013.

\bibitem{PhysRevE.64.046304}
Boris~A. Malomed, Bao-Feng Feng, and Takuji Kawahara.
\newblock Stabilized kuramoto-sivashinsky system.
\newblock {\em Phys. Rev. E}, 64:046304, Sep 2001.

\bibitem{Lim19}
Laurent Prouv\'{e}e and Juan L\'{\i}maco.
\newblock Local null controllability for a parabolic-elliptic system with local
  and nonlocal nonlinearities.
\newblock {\em Electron. J. Qual. Theory Differ. Equ.}, pages Paper No. 74, 31,
  2019.

\bibitem{Taka17}
Tak\'{e}o Takahashi.
\newblock Boundary local null-controllability of the {K}uramoto-{S}ivashinsky
  equation.
\newblock {\em Math. Control Signals Systems}, 29(1):Art. 2, 21, 2017.

\bibitem{RT}
Roger Temam.
\newblock {\em Infinite-dimensional dynamical systems in mechanics and
  physics}, volume~68 of {\em Applied Mathematical Sciences}.
\newblock Springer-Verlag, New York, second edition, 1997.

\bibitem{Zhou}
Z.~Zhou.
\newblock Observability estimate and null controllability for one-dimensional
  fourth order parabolic equation.
\newblock {\em Taiwanese J. Math.}, 16(6):1991--2017, 2012.

\end{thebibliography}

\end{document}